\documentclass[a4paper,12pt]{article}

\usepackage[left=2cm,right=2cm, top=2cm,bottom=3cm,bindingoffset=0cm]{geometry}

\usepackage{verbatim}
\usepackage{amsmath}
\usepackage{amsthm}
\usepackage{amssymb}
\usepackage{delarray}
\usepackage{cite}
\usepackage{hyperref}
\usepackage{mathrsfs}
\usepackage{tikz}
\usetikzlibrary{patterns}
\usepackage{caption}
\DeclareCaptionLabelSeparator{dot}{. }
\newcommand{\al}{\alpha}
\newcommand{\be}{\beta}
\newcommand{\ga}{\gamma}
\newcommand{\de}{\delta}
\newcommand{\la}{\lambda}
\newcommand{\om}{\omega}

\newcommand{\eps}{\varepsilon}
\newcommand{\vv}{\varphi}
\newcommand{\iy}{\infty}

\theoremstyle{plain}

\numberwithin{equation}{section}

\newtheorem{thm}{Theorem}[section]
\newtheorem{lem}[thm]{Lemma}
\newtheorem{prop}[thm]{Proposition}
\newtheorem{cor}[thm]{Corollary}

\theoremstyle{definition}

\newtheorem{alg}[thm]{Algorithm}

\newtheorem{prob}[thm]{Problem}
\newtheorem{df}[thm]{Definition}

\theoremstyle{remark}

\newtheorem{remark}[thm]{Remark}

\DeclareMathOperator*{\Res}{Res}

\DeclareMathOperator{\diag}{diag}

 \allowdisplaybreaks

\begin{document}

\begin{center}
{\Large\bf Reconstruction of higher-order differential operators  \\[0.2cm] by their spectral data}
\\[0.5cm]
{\bf Natalia P. Bondarenko}
\end{center}

\vspace{0.5cm}

{\bf Abstract.} This paper is concerned with inverse spectral problems for higher-order ($n > 2$) ordinary differential operators. We develop an approach to the reconstruction from the spectral data for a wide range of differential operators with either regular or distribution coefficients.
Our approach is based on the reduction of an inverse problem to a linear equation in the Banach space of bounded infinite sequences. This equation is derived in a general form that can be applied to various classes of differential operators. The unique solvability of the linear main equation is also proved. By using the solution of the main equation, we derive reconstruction formulas for the differential expression coefficients in the form of series and prove the convergence of these series for several classes of operators. The results of this paper can be used for constructive solution of inverse spectral problems and for investigation of their solvability and stability.

\medskip


\medskip


\vspace{1cm}

\section{Introduction} \label{sec:intr}

This paper is concerned with the inverse spectral theory for operators generated by the differential expression
\begin{align} \nonumber
\ell_n(y) := & y^{(n)} + \sum_{k = 0}^{\lfloor n/2\rfloor - 1} (\tau_{2k}(x) y^{(k)})^{(k)} \\ \label{defl} + & \sum_{k = 0}^{\lfloor (n-1)/2\rfloor - 1}  \bigl((\tau_{2k+1}(x) y^{(k)})^{(k+1)} + (\tau_{2k+1}(x) y^{(k+1)})^{(k)}\bigr), \: x \in (0,1),
\end{align}
where the notation $\lfloor a \rfloor$ means rounding down, the functions $\{ \tau_{\nu} \}_{\nu = 0}^{n-2}$ can be either integrable or distributional. Various aspects of spectral theory for such operators and related issues have been intensively studied in recent years (see, e.g., \cite{MS16, SS20, KM19, BH13, Pap16, BSHZ19, PB20, GN21, Bond21}). 
However, the general theory of inverse spectral problems for \eqref{defl} with arbitrary $n > 2$ has not been created yet. This paper aims to develop an approach to the reconstruction of the coefficients $\{ \tau_{\nu} \}_{\nu = 0}^{n-2}$ from the spectral data for a wide class of differential operators.

\subsection{Historical background}

% applications
Inverse problems of spectral analysis consist in the recovery of differential operators from their spectral information. Such problems arise in practice when one needs to determine certain physical parameters of a system from some measured data or to construct a model with desired properties. The majority of physical applications are concerned with linear differential operators of form \eqref{defl} with $n = 2, 3, 4$.

For $n = 2$, expression \eqref{defl} turns into the Sturm-Liouville (Schr\"odinger) operator
\begin{equation} \label{StL}
-\ell_2(y) = -y'' + q(x) y,
\end{equation}
which models string vibrations in classical mechanics, electron motion in quantum mechanics, and is
widely used in other branches of science and engineering. The third-order linear differential operators arise
in the inverse problem method for integration of the nonlinear Boussinesq equation (see \cite{DTT82, McK81}),
in mechanical problems of modeling thin membrane flow of viscous liquid and elastic beam
vibrations (see \cite{BP19} and references therein). Inverse spectral problems for the fourth-order linear differential operators attract much attention of scholars because of applications in mechanics and geophysics (see \cite{Barc74, McL76, PK97, CPS98,  Glad05, Mor15, BK15, JLX22} and references therein).

The classical results of the inverse problem theory have been obtained for the Sturm-Liouville operator \eqref{StL} with integrable potential $q(x)$
in 1950th by Marchenko, Levitan, and their followers (see \cite{Mar77, Lev84}). They have developed the transformation operator method, which reduces the nonlinear inverse Sturm-Liouville spectral problem to the linear Fredholm integral equation of the second kind. However, the transformation operator method appeared to be ineffective for the higher-order differential operators
\begin{equation} \label{ho}
y^{(n)} + \sum_{k = 0}^{n-2} p_k(x) y^{(k)}, \quad n > 2.
\end{equation}
Note that the differential expression \eqref{defl} can be transformed into \eqref{ho} in the case of sufficiently smooth coefficients $\{ \tau_{\nu} \}_{\nu = 0}^{n-2}$.

Thus, the development of inverse spectral theory for the higher-order operators \eqref{ho} required new approaches. Relying on the ideas of Leibenson \cite{Leib66, Leib71}, Yurko has created the method of spectral mappings. This method allowed him to construct inverse problem solutions for the higher-order differential operators \eqref{ho} with regular (integrable) coefficients on the half-line $x > 0$ and on a finite interval $x \in (0, T)$ (see \cite{Yur92, Yur02}). The case of Bessel-type singularities also was considered \cite{Yur93, Yur95}. Later on, the ideas of the method of spectral mappings were applied to a wide range of inverse spectral problems, e.g., to inverse problems for the first-order differential systems \cite{Yur05}, for differential operators on graphs \cite{Yur16}, for quadratic differential pencils \cite{BY12}. This method is based on the theory of analytic functions and mainly on the contour integration in the complex plane of the spectral parameter. The method of spectral mappings reduces a nonlinear inverse problem to a linear equation in a suitable Banach space. This space is constructed in different ways for different operator classes. In particular, for differential operators on a finite interval, the main equation is usually derived in the space $m$ of infinite bounded sequences. 
It is also worth mentioning that
an approach to inverse scattering problems for higher-order differential operators \eqref{ho} on the full line was developed by Beals et al \cite{Beals85, Beals88}.

During the last 20 years, the inverse problems are actively investigated for the second-order differential operators with distributional potentials (see, e.g., \cite{HM-sd, HM-2sp, HM-half, FIY08, MT09, SS10, HP12, Eckh14, Gul19, Bond21-AMP}). In particular, Hryniv and Mykytyuk \cite{HM-sd, HM-2sp, HM-half} transferred the transformation operator method to the Sturm-Liouville operators \eqref{StL} with potential $q(x)$ of class $W_2^{-1}(0,1)$ and so generalized the basic results of inverse problem theory to this class of operators. Note that the space $W_2^{-1}$ contains the Dirac $\de$-function and the Coulumb potential $\frac{1}{x}$, which are used for modeling particle interactions in quantum mechanics \cite{Alb05}. The method of spectral mappings has been extended to the Sturm-Liouville operators with potentials of $W_2^{-1}$ in \cite{FIY08, Bond21-AMP, Bond21-tamkang}. This opens the possibility of constructing the inverse spectral theory for higher-order differential operators with distribution coefficients. However, till now, only the first steps have been taken in this direction. In \cite{Bond21, Bond22}, the uniqueness of recovering the higher-order differential operators with distribution coefficients on a finite interval and on the half-line has been studied. The goals of this paper are to derive the linear main equation of the inverse problem, to prove its unique solvability, and to obtain reconstruction formulas for the coefficients $\{ \tau_{\nu} \}_{\nu = 0}^{n-2}$ of various classes.

\subsection{Problem statement and methods}

Our treatment of the differential expression \eqref{defl} is based on \textit{the regularization approach}. Namely, we will assume that the differential equation
\begin{equation} \label{eqv}
\ell_n(y) = \la y, \quad x \in (0, 1),
\end{equation}
where $\la$ is the spectral parameter,
can be equivalently transformed into the first-order system
\begin{equation} \label{sys}
Y'(x) = (F(x) + \Lambda) Y(x), \quad x \in (0, 1),
\end{equation}
where $Y(x)$ is a column vector function of size $n$, 
$\Lambda$ is the $(n \times n)$-matrix whose entry at the position $(n,1)$ equals $\la$ and all the other entries are zero, and $F(x) = [f_{k,j}(x)]_{k,j = 1}^n$ is a matrix function with the following properties:
\begin{equation} \label{propF}
\begin{array}{llll}
f_{k,j}(x) \equiv 0, \quad & k + 1 < j, \qquad
& f_{k,k+1}(x) \equiv 1, \quad & k = \overline{1,n-1}, \\
f_{k,k} \in L_2(0,1), \quad & k = \overline{1,n}, \qquad
& f_{k,j} \in L_1(0,1), \quad & k > j, \quad \mbox{trace}(F(x)) = 0.
\end{array}
\end{equation}
We denote the class of $(n \times n)$ matrix functions satisfying \eqref{propF} by $\mathfrak F_n$.

By using any matrix $F \in \mathfrak F_n$, one can define the quasi-derivatives
\begin{equation} \label{quasi}
y^{[0]} := y, \quad y^{[k]} = (y^{[k-1]})' - \sum_{j = 1}^k f_{k,j} y^{[j-1]}, \quad k = \overline{1,n},
\end{equation}
and the domain
$$
\mathcal D_F = \{ y \colon y^{[k]} \in AC[0,1], \, k = \overline{0, n-1} \}.
$$

\begin{df} \label{def:F}
A matrix function $F(x) \in \mathfrak F_n$ is called \textit{an associated matrix} of the differential expression $\ell_n(y)$ if $\ell_n(y) = y^{[n]}$ for any $y \in \mathcal D_F$. We call a function $y$ \textit{a solution} of equation \eqref{eqv} if $y \in \mathcal D_F$ and $y^{[n]} = \la y$, $x \in (0,1)$.
\end{df}

For a function $y \in \mathcal D_F$, introduce the notation
$\vec y(x) = \mbox{col} ( y^{[0]}(x), y^{[1]}(x), \ldots, y^{[n-1]}(x))$. Obviously, $y$ is a solution of equation \eqref{eqv} if and only if $Y = \vec y$ satisfies \eqref{sys}.

%The associated matrix $F(x)$ is constructed by the coefficients $\{ \tau_{\nu}\}_{\nu = 0}^{n-2}$ via a certain rule which depends on $n$ and the class of $\{ \tau_{\nu}\}_{\nu = 0}^{n-2}$.
The associated matrices for various classes of differential expressions $\ell_n(y)$ have been constructed, e.g., in \cite{MS16, MS19, KM19, VNS21, Bond22} (see also Subsections~\ref{sec:3}-\ref{sec:evenW} of this paper).
For example, for the differential expression $\ell_2(y) = y'' - \tau_0 y$, $\tau_0 \in W_2^{-1}(0,1)$, that is, $\tau_0 = \sigma'_0$, $\sigma_0 \in L_2(0,1)$, the associated matrix has the form (see \cite{SS99}):
$$
F(x) = \begin{bmatrix}
\sigma_0(x) & 1 \\ -\sigma_0^2(x) & -\sigma_0(x)
\end{bmatrix}.
$$

For the regular case $\tau_{\nu} \in L_1(0,1)$, $\nu = \overline{0,n-2}$, the construction of associated matrix $F(x)$ is well-known (see \cite{EM99} and Subsection~\ref{sec:evenL} of this paper). The regularization of even order ($n = 2m$) differential operators \eqref{defl} with distribution coefficients $\tau_{2k+j} \in W_2^{-(m-k-j)}(0,1)$, $k = \overline{0,m-1}$, $j = 0, 1$, has been obtained by Mirzoev and Shkalikov \cite{MS16}. Later on, the case of odd order $n$ was considered in \cite{MS19}. Vladimirov \cite{Vlad17} suggested a more general construction which, in particular, includes the both cases \cite{MS16, MS19}. It is worth mentioning that in \cite{MS16, MS19, Vlad17} the differential expressions of more general form than \eqref{defl} were studied, with the coefficients at $y^{(n)}$ and $y^{(n-1)}$ not necessarily equal $1$ and $0$, respectively. However, in this paper, we confine ourselves to the form \eqref{defl}, which is natural for studying the inverse problems \cite{Bond21, Bond22}.

In this paper, we assume that $\ell_n(y)$ is any differential expression that has an associated matrix in terms of Definition~\ref{def:F}. We do not impose any additional restrictions on $\{ \tau_{\nu} \}_{\nu = 0}^{n-2}$, since we are interested to formulate the abstract results which can be applied to various classes of differential operators. Certain restrictions on $\{ \tau_{\nu} \}_{\nu = 0}^{n-2}$ will be imposed below when necessary.

Let us proceed to the inverse problem formulation. Suppose that we have a differential expression of form \eqref{defl} and an associated matrix $F(x) = [f_{k,j}]_{k,j = 1}^n$.
By using the corresponding quasi-derivatives \eqref{quasi}, define the linear forms
\begin{equation} \label{defU}
\mathcal U_{s,a}(y) := y^{[p_{s,a}]}(a) + \sum_{j = 1}^{p_{s,a}} u_{s,j,a} y^{[j-1]}(a), \quad s = \overline{1, n}, \quad a = 0, 1,
\end{equation}
where $p_{s,a} \in \{ 0, \ldots, n-1 \}$, $p_{s,a} \ne p_{k,a}$ for $s \ne k$, and $u_{s,j,a}$ are some complex numbers. In addition, introduce the matrices $U_a = [u_{s,j,a}]_{s,j = 1}^n$, $u_{s,j,a} := \de_{j,p_{s,a} + 1}$ for $j > p_{s,a}$, $a = 0, 1$. Here and below, $\de_{j,k}$ is the Kronecker delta. We call by the problem $\mathcal L$ the triple $(F(x), U_0, U_1)$. Below we introduce various characteristics related to the problem $\mathcal L$.

Denote by $\{ C_k(x,\la) \}_{k = 1}^n$ and the solutions of equation~\eqref{eqv} satisfying the initial conditions
\begin{equation} \label{initC1}
\mathcal U_{s,0} (C_k) = \de_{s,k}, \quad s = \overline{1, n}.
\end{equation}
Equavalently, the $(n \times n)$-matrix function $C(x, \la) := [\vec C_k(x, \la)]_{k = 1}^n$ is the solution of the system \eqref{sys} with the initial condition $C(0, \la) = U_0^{-1}$. Therefore, the solutions $\{ C_k(x,\la) \}_{k = 1}^n$ are uniquely defined. Moreover, their quasi-derivatives $C_k^{[j]}(x, \la)$ are entire in $\la$ for each fixed $x \in [0,1]$, $k = \overline{1,n}$, $j = \overline{0,n-1}$.

It has been proved in \cite[Section 4]{Bond21} that, for all $\la \in \mathbb C$ except for a countable set, equation \eqref{eqv} has the so-called \textit{Weyl solutions} $\{ \Phi_k(x,\la) \}_{k = 1}^n$ satisfying the boundary conditions
\begin{equation} \label{bcPhi}
    \mathcal U_{s,0}(\Phi_k) = \de_{s,k}, \quad s = \overline{1, k}, \qquad
    \mathcal U_{s,1}(\Phi_k) = 0, \quad s = \overline{k+1,n},
\end{equation}
Define the matrix function $\Phi(x, \la) = [\vec \Phi_k(x, \la)]_{k = 1}^n$. The columns of the matrices $C(x, \la)$ and $\Phi(x, \la)$ form fundamental solution systems of \eqref{sys}. Consequently, the following relation holds: 
\begin{equation} \label{relM}
\Phi(x, \la) = C(x, \la) M(\la) 
\end{equation}
where the matrix function $M(\la)$ is called \textit{the Weyl matrix} of the problem $\mathcal L$ (see \cite{Bond21}).

The notion of Weyl matrix generalizes the notion of Weyl function for the second-order operators (see \cite{Mar77, Yur02}). Weyl functions and their generalizations play an important role in the inverse spectral theory for various classes of differential operators. In particular, Yurko \cite{Yur92, Yur93, Yur95, Yur02} has used the Weyl matrix as the main spectral characteristics for the reconstruction of the higher-order differential operators \eqref{ho} with regular coefficients. The analogous inverse problem for the differential expression of form \eqref{defl} can be formulated as follows.

\begin{prob} \label{prob:Weyl}
Given the Weyl matrix $M(\la)$, find the coefficients $\{ \tau_{\nu} \}_{\nu = 0}^{n-2}$.
\end{prob}

The uniqueness of Problem~\ref{prob:Weyl} solution has been proved in \cite{Bond21} for the Mirzoev-Shkalikov case: $n = 2m$, $\tau_{2k+j} \in W_2^{-(m-k-j)}(0,1)$ and $n = 2m+1$, $\tau_{2k+j} \in W_1^{-(m-k-j)}(0,1)$, $j = 0,1$. 
In \cite{Bond22}, the uniqueness of recovering the boundary condition coefficients from the Weyl matrix has been studied.

It has been shown in \cite[Section~4]{Bond21} that
the Weyl matrix $M(\la) = [M_{j,k}(\la)]_{j,k = 1}^n$ is unit lower-triangular, and its non-trivial entries have the form
\begin{equation} \label{Mjk}
M_{j,k}(\la) = -\frac{\Delta_{j,k}(\la)}{\Delta_{k,k}(\la)}, \quad 1 \le k < j \le n,
\end{equation}
where $\Delta_{k,k}(\la) := \det[\mathcal U_{s,1}(C_r)]_{s,r = k + 1}^n$ and $\Delta_{j,k}(\la)$ is obtained from $\Delta_{k,k}(\la)$ by the replacement of $C_j$ by $C_k$. The functions $C_r^{[s]}(1, \la)$, $r = \overline{1, n}$, $s = \overline{0,n-1}$, are entire analytic in $\la$, so do the functions $\Delta_{j,k}(\la)$, $1 \le k \le j \le n$. Hence, $M(\la)$ is meromorphic in $\la$, and the poles of the $k$-th column of $M(\la)$ coincide with the zeros of $\Delta_{k,k}(\la)$. At the same time, the zeros of the entire functions $\Delta_{j,k}(\la)$ , $1 \le k \le j \le n$, coincide with the eigenvalues of some boundary value problems for equation \eqref{eqv}, and the inverse problem by the Weyl matrix (Problem~\ref{prob:Weyl}) is related to the inverse problem by $\frac{n(n+1)}{2}$ spectra (see \cite{Bond21} for details).

We will say that the problem $\mathcal L$ belongs to the class $W$ if all the zeros of $\Delta_{k,k}(\la)$ are simple for $k = \overline{1,n-1}$. Then, in view of \eqref{Mjk}, the poles of $M(\la)$ are simple. In general, the function $\Delta_{k,k}(\la)$ can have at most finite number of multiple zeros. The latter case can be treated by developing the methods of Buterin et al \cite{But07, BSY13}, who considered the non-self-adjoint Sturm-Liouville operators ($n = 2$) with regular potentials.
However, the case of multiple zeros is much more technically complicated, so, in this paper, we always assume that $\mathcal L \in W$.

Denote by $\Lambda$ the set of the Weyl matrix poles.
Consider the Laurent series
$$
M(\la) = \frac{M_{\langle -1 \rangle}(\la_0)}{\la - \la_0} + M_{\langle 0 \rangle}(\la_0) + M_{\langle 1 \rangle}(\la_0)(\la - \la_0) + \dots, \quad \la_0 \in \Lambda.
$$
Denote
\begin{equation} \label{defN}
\mathcal N(\la_0) := [M_{\langle 0 \rangle}(\la_0)]^{-1} M_{\langle -1\rangle}(\la_0), \quad \la_0 \in \Lambda,
\end{equation}
We call the collection $\{ \la_0, \mathcal N(\la_0) \}_{\la_0 \in \Lambda}$ \textit{the spectral data} of the problem $\mathcal L$. Obviously, the spectral data are uniquely specified by the Weyl matrix $M(\la)$, so Problem~\ref{prob:Weyl} can be reduced to the following problem.

\begin{prob} \label{prob:sd-coef}
Given the spectral data $\{ \la_0, \mathcal N(\la_0) \}_{\la_0 \in \Lambda}$, find the coefficients $\{ \tau_{\nu} \}_{\nu = 0}^{n-2}$. 
\end{prob}

It is more convenient to study the reconstruction question for Problem~\ref{prob:sd-coef}.
It is worth mentioning that, in fact, the Weyl matrix and the spectral data can be constructed according to the above definitions for any matrix function $F(x)$ of class $\mathfrak F_n$, not necessarily associated with any differential expression of form \eqref{defl}. But, in general, the matrix $F(x)$ is not uniquely specified by the Weyl matrix (see Example~4.5 in \cite{Bond22}). Therefore, in this paper, the solution of Problem~\ref{prob:sd-coef} is divided into the two steps:
$$
\{ \la_0, \mathcal N(\la_0) \}_{\la_0 \in \Lambda} \: \stackrel{(1)}{\to} \: \{ \Phi_k(x, \la) \}_{k = 1}^n \: \stackrel{(2)}{\to} \: \{ \tau_{\nu} \}_{\nu = 0}^{n-2}.
$$

The recovery of the Weyl solutions $\{ \Phi_k(x, \la) \}_{k = 1}^n$ from the spectral data is studied for a matrix $F(x)$ of general form, and then reconstruction formulas are derived for $\{ \tau_{\nu} \}_{\nu = 0}^{n-2}$ of certain classes.

For a fixed $F \in \mathfrak F_n$, we define the quasi-derivatives \eqref{quasi}, the expression $\ell_n(y) := y^{[n]}$, the problem $\mathcal L = (F(x), U_0, U_1)$, its spectral data
$\{ \la_0, \mathcal N(\la_0) \}_{\la_0 \in \Lambda}$ as above, and focus on the following auxiliary problem.

\begin{prob} \label{prob:sd}
Given the spectral data $\{ \la_0, \mathcal N(\la_0) \}_{\la_0 \in \Lambda}$, find the Weyl solutions $\{ \Phi_k(x, \la) \}_{k = 1}^n$.
\end{prob}

Let us briefly describe the method of solution.
Along with $\mathcal L$, we
consider another problem $\tilde {\mathcal L} = (\tilde F(x), \tilde U_0, \tilde U_1)$ of the same form but with different coefficients. Similarly to $\Phi(x, \la)$, define $\tilde \Phi(x, \la)$ for $\tilde {\mathcal L}$. An important role in our analysis is played by \textit{the matrix of spectral mappings}:
$$
    \mathcal P(x, \la) = \Phi(x, \la) [\tilde \Phi(x, \la)]^{-1}.
$$

For each fixed $x \in [0,1]$, the matrix function $\mathcal P(x, \la)$ is meromorpic in $\la$ with poles at the eigenvalues $\Lambda \cup \tilde \Lambda$. The method is based on the integration of some functions by a special family of contours enclosing these eigenvalues. Applying the Residue theorem, we derive an infinite system of linear equations. Further, that system is transformed into a linear equation in the Banach space $m$ of infinite bounded sequences. The main equation of the inverse problem has the form
$$
    (\mathbf{I} - \tilde R(x)) \psi(x) = \tilde \psi(x), \quad x \in [0,1],
$$
where,
for each fixed $x \in [0,1]$, $\psi(x)$ and $\tilde \psi(x)$ are elements of $m$, $\tilde R(x)$ is a linear compact operator in $m$, and $\mathbf{I}$ is the unit operator. The element $\tilde \psi(x)$ and the operator $\tilde R(x)$ are constructed by the model problem $\tilde {\mathcal L}$ and by the spectral data $\{ \la_0, \mathcal N(\la_0) \}_{\la_0 \in \Lambda}$, $\{ \tilde \la_0, \tilde{\mathcal N}(\tilde \la_0) \}_{\tilde \la_0 \in \tilde \Lambda}$ of the two problems $\mathcal L$, $\tilde {\mathcal L}$, respectively, while the unknown element $\psi(x)$ is related to the desired functions $\{ \Phi_k(x, \la) \}_{k = 1}^n$. We prove that the operator $(\mathbf{I} - \tilde R(x))$ has the bounded inverse, and so the main equation is uniquely solvable (see Theorem~\ref{thm:main}). This implies the uniqueness of solution for Problem~\ref{prob:sd}. Using the main equation, we obtain a constructive procedure for solving Problem~\ref{prob:sd} (see Algorithm~\ref{alg:1}). These results can be applied to a wide range of differential operators \eqref{defl} with associated matrices of class $\mathfrak F_n$.

Further, by using the solution of the main equation, we derive reconstruction formulas for $\{ \tau_{\nu} \}_{\nu = 0}^{n-2}$. We describe the general idea and then apply it to the certain classes of operators:

\medskip

(i) $n = 3$, $\tau_1 \in L_2(0,1)$, $\tau_0 \in W_2^{-1}(0,1)$.

\smallskip

(ii) $n$ is even, $\tau_{\nu} \in L_2(0,1)$, $\nu = \overline{0,n-2}$.

\smallskip

(iii) $n$ is even, $\tau_{\nu} \in W_2^{-1}(0,1)$, $\nu = \overline{0,n-2}$.

\medskip

We obtain the uniqueness theorems and constructive algorithms for solving Problem~\ref{prob:sd-coef} for the cases (i)-(iii). Note that, although the functions $\tau_{\nu}$ in the case (ii) are regular, this case has less smoothness than the one considered by Yurko \cite{Yur02}. 

The reconstruction formulas have the form of series, and the main difficulties in our analysis are related to studying the convergence of those series. These difficulties increase for the case of non-smooth and/or distribution coefficients.
In order to prove the series convergence, we use the Birkhoff-type solutions constructed by Savchuk and Shkalikov \cite{SS20} and
the precise asymptotic formulas for the spectral data obtained in \cite{Bond22-asympt}. For the cases (ii) and (iii), we reconstruct the functions $\tau_{\nu}$ step-by-step for $\nu = n-2, n-3, \ldots, 1, 0$. The similar approach can be used in the case of odd $n$, which requires technical modifications.

By using the reconstruction formulas, one can develop numerical methods for solving inverse spectral problems (see \cite{IY08} for the second-order case). However, this issue requires an additional work. In this paper, we obtain theoretical algorithms, which in the future can be used for investigation of existence and stability of the inverse problem solution.

It is worth mentioning that our method of inverse problem solution is the first one for higher-order differential operators with distribution coefficents.
The obtained main equation and reconstruction formulas generalize the results of \cite{Bond21-tamkang} for the Sturm-Liouville operators with distribution potential. The other methods applied to the second-order operators (see, e.g., \cite{HM-sd, SS10}), to the best of the author's knowledge, appear to be ineffective for higher orders.

\medskip

The paper is organized as follows. In Section~\ref{sec:prelim}, we provide preliminaries and study the properties of the spectral data. Section~\ref{sec:main} is devoted to the contour integration and to the derivation of the main equation of the inverse problem in a Banach space. The unique solvability of the main equation is also proved. As a result, an algorithm for solving the auxiliary Problem~\ref{prob:sd} is obtained for arbitrary $F \in \mathfrak F_n$. In Section~\ref{sec:rec}, we derive the reconstruction formulas for the coefficients $\{ \tau_{\nu} \}_{\nu = 0}^{n-2}$ and study the convergence of the obtained series. Section~\ref{sec:concl} contains a brief summary of the main results.

\section{Preliminaries} \label{sec:prelim}

Throughout the paper, we use the following \textbf{notations}.

\begin{enumerate}
    \item $I$ is the $(n \times n)$ unit matrix, $e_k$ is the $k$-th column of $I$, $k = \overline{1, n}$.
    \item The sign $T$ denotes the matrix transpose.
    \item $\de_{k,j} = \begin{cases}
    1, \quad k = j, \\ 0, \quad k \ne j.
    \end{cases}$
    \item $J := [(-1)^{k+1} \de_{k,n-j+1}]_{k,j = 1}^n$, $J_a := [(-1)^{p_{k,a}^{\star}}\de_{k,n-j+1}]_{k,j = 1}^n$, where $p_{k,a}^{\star} := n-1-p_{k,a}$, $a = 0, 1$.
    \item If for $\la \to \la_0$
    $$
    A(\la) = \sum_{k = -q}^p a_k(\la - \la_0)^k + o((\la-\la_0)^p),
    $$
    then
    $$
        [A(\la)]_{|\la = \la_0}^{\langle k \rangle} = A_{\langle k \rangle}(\la_0) := a_k.
    $$
    \item The notations $\lfloor x \rfloor$ and $\lceil x \rceil$ are used for rounding a real number $x$ down and up, respectively.
    \item The binomial coefficients are denoted by $C_n^k = \dfrac{n!}{k!(n-k)!}$.
    \item Along with $\mathcal L$, we will consider the problems $\tilde {\mathcal L}$, $\mathcal L^{\star}$, $\tilde {\mathcal L}^{\star}$ of the same form but with different coefficients. We agree that, if a symbol $\ga$ denotes an object related to $\mathcal L$, then the symbols $\tilde \ga$, $\ga^{\star}$, $\tilde \ga^{\star}$ will denote the analogous objects related to $\tilde {\mathcal L}$, $\mathcal L^{\star}$, $\tilde {\mathcal L}^{\star}$, respectively. Note that the quasi-derivatives for the problems $\tilde {\mathcal L}$, $\mathcal L^{\star}$, $\tilde {\mathcal L}^{\star}$ are defined by using the  matrices $\tilde F(x)$, $F^{\star}(x)$, $\tilde F^{\star}(x)$, respectively, which may be different from $F(x)$.
    \item The notation $y^{[k]}$ is used for quasi-derivatives defined by \eqref{quasi} (or analogously by using the entries of $\tilde F(x)$, $F^{\star}(x)$, or $\tilde F^{\star}(x)$). The notation $\vec y(x)$ is used for the column vector of the quasi-derivatives $y^{[0]}(x)$, $y^{[1]}(x)$, \dots, $y^{[n-1]}(x)$.
    \item In estimates, the symbol $C$ is used for various positive constants independent of $x$, $l$, $k$, etc.
    \item $a \stackrel{if \: (condition)}{\times} b = \begin{cases}
    a b, \quad \text{if (condition) holds}, \\
    a, \quad \text{otherwise}.
    \end{cases}$.
\end{enumerate}

In Subsection~\ref{sec:star}, we define an auxiliary problem $\mathcal L^{\star} = (F^{\star}(x), U_0^{\star}, U_1^{\star})$ and study its properties. In Subsection~\ref{sec:sd}, the properties of the spectral data $\{ \la_0, \mathcal N(\la_0) \}_{\la_0 \in \Lambda}$ are investigated. 

\subsection{Problems $\mathcal L$ and $\mathcal L^{\star}$} \label{sec:star}

For a matrix $F \in \mathfrak F_n$, define the matrix $F^{\star}(x) = [f_{k,j}^{\star}(x)]_{k,j = 1}^n$ as follows:
\begin{equation} \label{fstar}
f_{k,j}^{\star}(x) := (-1)^{k+j+1} f_{n-j+1, n-k+1}(x).
\end{equation}
Obviously, $F^{\star} \in \mathfrak F_n$.

Let $F(x)$ be a fixed matrix function of class $\mathfrak F_n$. Suppose that $y \in \mathcal D_F$ and $z \in \mathcal D_{F^{\star}}$, the quasi-derivatives for $y$ are defined via \eqref{quasi} by using the elements of $F(x)$, the quasi-derivatives for $z$ are defined as
\begin{equation} \label{quasiz}
z^{[0]} := z, \quad z^{[k]} = (z^{[k-1]})' - \sum_{j = 1}^k f^{\star}_{k,j} z^{[j-1]}, \quad k = \overline{1,n},
\end{equation}
and
$$
\mathcal D_{F^{\star}} := \{ z \colon z^{[k]} \in AC[0,1], \, k = \overline{0, n-1} \}.
$$

Define 
\begin{gather*}
\ell_n(y) := y^{[n]}, \quad \ell_n^{\star}(z) := (-1)^n z^{[n]}, \quad
\langle z, y \rangle := \sum_{j = 0}^{n-1} (-1)^j z^{[j]} y^{[n-j-1]}.
\end{gather*}

\begin{lem}
The following relation holds:
\begin{equation} \label{wron1}
\frac{d}{dx} \langle z, y \rangle = z \ell_n(y) - y \ell_n^{\star}(z).
\end{equation}
\end{lem}

\begin{proof}
Differentiation implies
\begin{equation} \label{sm1}
\frac{d}{dx} \langle z, y \rangle = \sum_{j = 0}^{n-1} (-1)^j (z^{[j]})' y^{[n-j-1]} + \sum_{j = 0}^{n-1} (-1)^j z^{[j]} (y^{[n-j-1]})'.
\end{equation}
From \eqref{quasiz} and \eqref{quasi}, we obtain
$$
    (z^{[j]})'  = z^{[j+1]} + \sum_{s = 1}^{j+1} f_{j+1, s}^{\star} z^{[s-1]}, \quad
    (y^{[n-j-1]})' = y^{[n-j]} + \sum_{s = 1}^{n-j} f_{n-j,s} y^{[s-1]}.
$$
Substituting the latter relations into \eqref{sm1}, we get
\begin{align*}
\frac{d}{dx} \langle z, y \rangle
= & \sum_{j = 0}^{n-1} (-1)^j y^{[n-j]} z^{[j]} + \sum_{j = 0}^{n-1} (-1)^j \sum_{s = 1}^{n-j} f_{n-j,s} y^{[s-1]} z^{[j]} \\ & + \sum_{j = 0}^{n-1} (-1)^j y^{[n-j-1]} z^{[j+1]} + \sum_{j = 0}^{n-1} (-1)^j \sum_{s = 1}^{j+1} f_{j+1,s}^{\star} y^{[n-j-1]} z^{[s-1]}.
\end{align*}
Note that
\begin{align*}
\sum_{j = 0}^{n-1} (-1)^j y^{[n-j]} z^{[j]} + \sum_{j = 0}^{n-1} (-1)^j y^{[n-j-1]} z^{[j+1]} & = y^{[n]} z + (-1)^{n-1} y z^{[n]}, \\
\sum_{j = 0}^{n-1} (-1)^j \sum_{s = 1}^{n-j} f_{n-j,s} y^{[s-1]} z^{[j]} & = \sum_{1 \le s \le j \le n} (-1)^{s+1} f_{n-s+1, n-j+1}y^{[n-j]} z^{[s-1]}, \\
\sum_{j = 0}^{n-1} (-1)^j \sum_{s = 1}^{j+1} f_{j+1,s}^{\star} y^{[n-j-1]} z^{[s-1]} & = \sum_{1 \le s \le j \le n} (-1)^{j+1} f_{j,s}^{\star} y^{[n-j]} z^{[s-1]}.
\end{align*}
Taking \eqref{fstar} into account, we arrive at \eqref{wron1}.
\end{proof}

If $y$ and $z$ satisfy the relations $\ell_n(y) = \la y$ and $\ell_n^{\star}(z) = \mu z$, respectively, then \eqref{wron1} readily implies
\begin{equation} \label{wron2}
\frac{d}{dx} \langle z, y \rangle = (\la - \mu) y z.
\end{equation}

Define $\vec y(x) = \mbox{col} ( y^{[0]}(x), y^{[1]}(x), \ldots, y^{[n-1]}(x))$ and $\vec z(x) = \mbox{col} ( z^{[0]}(x), z^{[1]}(x), \ldots, z^{[n-1]}(x))$ by using the corresponding quasi-derivatives \eqref{quasi} and \eqref{quasiz}, and the matrix $J := [(-1)^{k+1}\de_{k,n-j+1}]_{k,j = 1}^n$. Then
\begin{equation} \label{wrona1}
\langle z, y \rangle_{|x = a} = [\vec z(a)]^T J \vec y(a).
\end{equation}

For $a = 0, 1$, let $U_a = [u_{s,j,a}]_{s,j = 1}^n$ be an $(n \times n)$ matrix such that $u_{s,j,a} = \de_{j, p_{s,a} + 1}$ for $j > p_{s,a}$, where $p_{s,a} \in \{ 0, \ldots, n-1\}$, $p_{s,a} \ne p_{k,a}$ for $s \ne k$. The matrices $U_a$ define the linear forms $\mathcal U_{s,a}$ via \eqref{defU}.

Along with $U_a$, consider the matrices 
\begin{equation} \label{defUs}
U_a^{\star} := [J_a^{-1} U_a^{-1} J]^T, \quad a = 0, 1,
\end{equation}
where $J_a = [(-1)^{p_{k,a}^{\star}}\de_{k,n-j+1}]_{k,j = 1}^n$, $p_{k,a}^{\star} := n - 1 - p_{n-k+1, a}$. The matrices $U_a^{\star}$, $a = 0, 1$, generate the linear forms
$$
\mathcal U_{s,a}^{\star}(z) = z^{[p_{s,a}^{\star}]}(a) + \sum_{j = 1}^{p_{s,a}^{\star}} u_{s,j,a}^{\star} z^{[j-1]}(a), \quad s = \overline{1, n}, \quad a = 0,1.
$$

The matrices $U_a^{\star}$ are chosen is such a way that the following relation holds:
\begin{equation} \label{wrona2}
\langle z, y \rangle_{|x = a} = \sum_{s = 1}^n (-1)^{p_{s,a}^{\star}} \mathcal U_{s,a}^{\star}(z) \mathcal U_{n-s+1,a}(y)
\end{equation}
for any $y \in \mathcal D_F$, $z \in \mathcal D_{F^{\star}}$. Indeed, the right-hand side of \eqref{wrona2} can be represented in the matrix form
$$
[U_a^{\star} \vec z(a)]^T J_a U_a \vec y(a),
$$
Taking \eqref{wrona1} and \eqref{defUs} into account, we arrive at \eqref{wrona2}.

Consider the problems $\mathcal L = (F(x), U_0, U_1)$ and $\mathcal L^{\star} = (F^{\star}(x), U_0^{\star}, U_1^{\star})$. For $\mathcal L$, the matrix functions $C(x, \la)$, $\Phi(x, \la)$, and $M(\la)$ were defined in the Introduction. For $\mathcal L^{\star}$, similarly denote by $\{ C_k^{\star}(x, \la) \}_{k = 1}^n$ and $\{ \Phi_k^{\star}(x, \la) \}_{k = 1}^n$ the solutions of equation $\ell_n^{\star}(z) = \la z$, $x \in (0, 1)$, satisfying the conditions
\begin{gather} \nonumber
\mathcal U^{\star}_{s,0} (C_k^{\star}) = \de_{s,k}, \quad s = \overline{1, n}, \\ \label{bcPhis}
    \mathcal U_{s,0}^{\star}(\Phi_k^{\star}) = \de_{s,k}, \quad s = \overline{1, k}, \qquad
    \mathcal U_{s,1}^{\star}(\Phi_k^{\star}) = 0, \quad s = \overline{k+1,n}.
\end{gather}
Put $C^{\star}(x, \la) := [\vec C_k^{\star}(x, \la)]_{k = 1}^n$, $\Phi^{\star}(x, \la) := [\vec \Phi_k^{\star}(x, \la)]_{k = 1}^n$.
Then, the relation 
\begin{equation} \label{relMs}
\Phi^{\star}(x, \la) = C^{\star}(x, \la) M^{\star}(\la)
\end{equation}
holds, where $M^{\star}(\la)$ is the Weyl matrix of the problem $\mathcal L^{\star}$.

\begin{lem} \label{lem:M}
The following relations hold:
\begin{gather} \label{MJM}
[M^{\star}(\la)]^T J_0 M(\la) = J_0, \\ \label{PJP}
[\Phi^{\star}(x, \la)]^T J \Phi(x, \la) = J_0.
\end{gather}
\end{lem}

\begin{proof}
The initial conditions \eqref{initC1} are equivalent to $U_0 C(0, \la) = I$. % Unit matrix?
Using \eqref{relM}, we get $M(\la) = U_0 \Phi(0, \la)$. Similarly, $M^{\star}(\la) = U_0^{\star} \Phi^{\star}(0,\la)$. Hence
\begin{gather} \nonumber
A(\la) := [M^{\star}(\la)]^T J_0 M(\la) = [U_0^{\star} \Phi^{\star}(0,\la)]^T J_0 U_0 \Phi(0,\la), \quad A(\la) = [A_{k,j}(\la)]_{k,j = 1}^n,\\ \label{Akj}
A_{k,j}(\la) = [U_0^{\star} \vec \Phi^{\star}_k(0,\la)]^T J_0 U_0 \vec \Phi_j(0, \la) = \sum_{s = 1}^n (-1)^{p_{s,0}^{\star}} \mathcal U_{s,0}^{\star}(\Phi_k^{\star}) \mathcal U_{n-s+1,0}(\Phi_j).
\end{gather}

On the one hand, using \eqref{Akj}, \eqref{bcPhi}, and \eqref{bcPhis}, we get $A_{k,j}(\la) = 0$ if $k + j > n + 1$ and $A_{k,j}(\la) = (-1)^{p_{k,0}^{\star}}$ if $k + j = n + 1$. On the other hand, \eqref{wrona2} and \eqref{Akj} imply $A_{k,j}(\la) = \langle \Phi_k^{\star}, \Phi_j \rangle_{|x = 0}$. It follows from \eqref{wron2} that $\langle \Phi_k^{\star}, \Phi_j \rangle$ does not depend on $x$. Consequently,
$$
\langle \Phi_k^{\star}, \Phi_j \rangle_{|x = 0} = \langle \Phi_k^{\star}, \Phi_j \rangle_{|x = 1} = \sum_{s = 1}^n (-1)^{p_{s,1}^{\star}} \mathcal U_{s,1}^{\star}(\Phi_k^{\star}) \mathcal U_{n-s+1,1}(\Phi_j).
$$
Using the boundary conditions \eqref{bcPhi} and \eqref{bcPhis} at $x = 1$, we conclude that $A_{k,j}(\la) = 0$ if $k + j < n + 1$. Thus, $A(\la) = J_0$ and \eqref{MJM} is proved.

Using the relation $A_{k,j}(\la) = \langle \Phi_k^{\star}, \Phi_j\rangle$ for $k,j = \overline{1,n}$ and \eqref{wrona1}, we obtain
$$
A(\la) = [\Phi^{\star}(x,\la)]^T J \Phi(x, \la).
$$
This implies \eqref{PJP}.
\end{proof}

\subsection{Spectral data} \label{sec:sd}

Consider the Weyl matrix $M(\la)$ of the problem $\mathcal L = (F(x), U_0, U_1)$, where $F \in \mathfrak F_n$. Recall that the poles of the $k$-th column of $M(\la)$ coincide with the zeros of $\Delta_{k,k}(\la) = \det[\mathcal U_{s,1}(C_r)]_{s,r = k+1}^n$. One can easily show that the zeros of $\Delta_{k,k}(\la)$ coincide with the eigenvalues 
of the following boundary value problem $\mathcal L_k$:
$$
\ell_n(y) = \la y, \quad x \in (0,1), \qquad \mathcal U_{s,0}(y) = 0, \quad s = \overline{1,k}, \qquad
\mathcal U_{s,1}(y) = 0, \quad s = \overline{k+1,n}.
$$

By virtue of Theorem~1.1 in \cite{Bond22-asympt}, the spectrum of $\mathcal L_k$ is a countable set of eigenvalues $\Lambda_k := \{ \la_{l,k} \}_{l \ge 1}$  having the following asymptotics (counting with multiplicities):
\begin{equation} \label{asymptla}
\la_{l,k} = (-1)^{n-k} \left( \frac{\pi}{\sin\tfrac{\pi k}{n}} (l + \chi_k + \varkappa_{l,k}) \right)^n,  
\end{equation}
where $\{ \varkappa_{l,k} \} \in l_2$ and $\chi_k$ are constants which depend only on $n$, $k$, and $\{ p_{s,a} \}$.
Hence, for a fixed $k \in \{ 1, \ldots, n-1 \}$ and sufficiently large $l$, the eigenvalues $\la_{l,k}$ are simple. 

Assume that $\mathcal L \in W$, that is, all the zeros of $\Delta_{k,k}(\la)$ are simple for $k = \overline{1,n-1}$.
Then, in view of \eqref{Mjk} and \eqref{MJM}, the poles of $M(\la)$ and $M^{\star}(\la)$ are simple. It follows from \eqref{relM} and \eqref{relMs} that the matrix functions $\Phi(x, \la)$ and $\Phi^{\star}(x, \la)$ for each fixed $x \in [0,1]$ also have only simple poles.  

Denote $\Lambda := \bigcup_{k = 1}^{n-1} \Lambda_k$. Similarly to $\mathcal N(\la_0)$, denote
\begin{equation} \label{defNs}
\mathcal N^{\star}(\la_0) := [M^{\star}_{\langle 0 \rangle}(\la_0)]^{-1} M^{\star}_{\langle -1\rangle}(\la_0), \quad \la_0 \in \Lambda.
\end{equation}
For $\la_0 \not\in \Lambda$, we mean that $\mathcal N(\la_0) = \mathcal N^{\star}(\la_0) = 0$.

Let us study some properties of the matrices $\mathcal N(\la_0)$ and $\mathcal N^{\star}(\la_0)$.
Denote by $\phi(x, \la)$ the first row of the matrix function $\Phi(x, \la)$: $\phi(x, \la) = e_1^T \Phi(x, \la) = [\Phi_k(x, \la)]_{k = 1}^n$.

\begin{lem} \label{lem:N1}
The following relations hold for each $\la_0 \in \Lambda$: $\mathcal N^2(\la_0) = 0$,
\begin{gather} \label{relN1}
 [\mathcal N^{\star}(\la_0)]^T = - J_0 \mathcal N(\la_0) J_0^{-1}, \\ \label{relNPhi}
    \Phi_{\langle -1 \rangle}(x, \la_0) = \Phi_{\langle 0 \rangle}(x, \la_0) \mathcal N(\la_0), \quad
    \Phi^{\star}_{\langle -1 \rangle}(x, \la_0) = \Phi^{\star}_{\langle 0 \rangle}(x, \la_0) \mathcal N^{\star}(\la_0), \\ \label{lnphi}
    \ell_n(\phi_{\langle 0 \rangle}(x, \la_0)) = \la_0 \phi_{\langle 0 \rangle}(x, \la_0) + \phi_{\langle 0 \rangle}(x, \la_0) \mathcal N(\la_0).
\end{gather}
\end{lem}

\begin{proof}
The relation \eqref{MJM} implies
\begin{gather} \label{smM1}
    [M(\la)]^{-1} = J_0^{-1} [M^{\star}(\la)]^T J_0, \\ \label{smM2}
    M(\la) J_0^{-1} [M^{\star}(\la)]^T = J_0^{-1}.
\end{gather}
It follows from \eqref{smM2} that
\begin{gather} \label{smM3}
    M_{\langle -1 \rangle}(\la_0) J_0^{-1} [M_{\langle -1 \rangle}^{\star}(\la_0)]^T = 0, \\ \label{smM4}
    M_{\langle 0 \rangle}(\la_0) J_0^{-1} [M_{\langle -1 \rangle}^{\star}(\la_0)]^T + M_{\langle -1 \rangle}(\la_0) J_0^{-1} [M_{\langle 0 \rangle}^{\star}(\la_0)]^T = 0.
\end{gather}
Using \eqref{defN}, \eqref{defNs}, and \eqref{smM4}, we obtain \eqref{relN1}. Multiplying \eqref{relN1} by $\mathcal N(\la_0)$ and using \eqref{smM3}, we derive
$$
\mathcal N(\la_0) J_0^{-1} [\mathcal N^{\star}(\la_0)]^T = -\mathcal N^2(\la_0) J_0^{-1} = 0.
$$
Hence $\mathcal N^2(\la_0) = 0$.

Using \eqref{relM} and \eqref{smM1}, we obtain
$$
C(x, \la) = \Phi(x, \la) [M(\la)]^{-1} = \Phi(x, \la) J_0^{-1} [M^{\star}(\la)]^T J_0.
$$
Since $C(x, \la)$ is entire in $\la$ for each fixed $x \in [0,1]$, then we get
\begin{equation} \label{smM5}
    \Phi_{\langle 0 \rangle}(x, \la_0) J_0^{-1} [M^{\star}_{\langle -1\rangle}(\la_0)]^T J_0 +
    \Phi_{\langle -1 \rangle}(x, \la_0) J_0^{-1} [M^{\star}_{\langle 0\rangle}(\la_0)]^T J_0 = 0, \quad \la_0 \in \Lambda.
\end{equation}
Using \eqref{smM5} and \eqref{defNs}, we derive
$$
 \Phi_{\langle 0 \rangle}(x, \la_0) J_0^{-1} [\mathcal N^{\star}(\la_0)]^T J_0  + \Phi_{\langle -1 \rangle}(x, \la_0) = 0.
$$
Taking \eqref{relN1} into account, we arrive at the first relation in \eqref{relNPhi}. The second one is similar.

It follows from the relation 
$\ell_n(\phi(x, \la)) = \la \phi(x, \la)$
that
\begin{align*}
    & \ell_n(\phi_{\langle -1 \rangle}(x, \la_0)) = \la_0 \phi_{\langle -1 \rangle}(x, \la_0), \\
    & \ell_n(\phi_{\langle 0 \rangle}(x, \la_0)) = \la_0 \phi_{\langle 0 \rangle}(x, \la_0) + \phi_{\langle -1 \rangle}(x, \la_0).
\end{align*}
Using \eqref{relNPhi}, we arrive at \eqref{lnphi}.
\end{proof}

Consider the entries of the matrix $\mathcal N(\la_0) = [\mathcal N_{k,j}(\la_0)]_{k,j = 1}^n$. Since $M(\la)$ is unit lower-triangular, we have $\mathcal N_{k,j}(\la_0) = 0$ for all $k \le j$, $\la_0 \in \Lambda$.
The structural properties of $\mathcal N(\la_0)$ are described by the following lemma.

\begin{lem} \label{lem:N2}
(i) If $\la_0 \not\in \Lambda_k$, then $\mathcal N_{s, j}(\la_0) = 0$, $s = \overline{k+1,n}$, $j = \overline{1,k}$. 

(ii) If $\la_0 \in \Lambda_s$ for $s = \overline{\nu + 1, k-1}$, $\la_0 \not\in \Lambda_{\nu}$, $\la_0 \not\in \Lambda_k$, $1 \le \nu + 1 < k \le n$, then $\mathcal N_{k, \nu + 1}(\la_0) \ne 0$. (Here $\Lambda_0 = \Lambda_n = \varnothing$).
\end{lem}

\begin{proof}
This lemma is proved similarly to Lemma~2.3.1 in \cite{Yur02}, so we outline the proof briefly. If $\la_0 \not\in \Lambda_k$, then $\Phi_{k,\langle-1\rangle}(x, \la_0) = 0$. On the other hand, it follows from \eqref{relNPhi} that
$$
\Phi_{k,\langle-1\rangle}(x, \la_0) = \sum_{s = k+1}^n \mathcal N_{s,k}(\la_0) \Phi_{s,\langle 0 \rangle}(x, \la_0).
$$
Applying the linear forms $\mathcal U_{s,0}$ to this relation for $s = \overline{k+1,n}$, we conclude that $\mathcal N_{s,k}(\la_0) = 0$, $s = \overline{k+1,n}$. Thus, the assertion (i) is proved for $j = k$. The proof for $j = k-1, \ldots, 2, 1$ can be obtained by induction.

In order to prove (ii), we suppose that $\Delta_{\nu,\nu}(\la_0) \ne 0$, $\Delta_{s,s}(\la_0) = 0$ for $s = \overline{\nu+1,k-1}$. Then, it can be shown that $\mathcal U_{s,1}(\Phi_{s,\langle 0 \rangle}(x, \la_0)) \ne 0$, $s = \overline{\nu + 2, k-1}$ and $\Phi_{\nu + 1, \langle -1 \rangle}(x, \la_0) \not\equiv 0$. Suppose that $\mathcal N_{k,\nu+1}(\la_0) = 0$. Consequently, \eqref{relNPhi} implies
$$
\Phi_{\nu+1,\langle-1\rangle}(x, \la_0) = \sum_{s = \nu+2}^{k-1} \mathcal N_{s,\nu+1}(\la_0) \Phi_{s,\langle 0 \rangle}(x, \la_0).
$$
Applying the linear forms $\mathcal U_{s,1}$ for $s = \overline{\nu+2,k-1}$, we conclude that $N_{s,\nu+1}(\la_0) = 0$, $s = \overline{\nu+2,k-1}$, and so $\Phi_{\nu+1,\langle-1\rangle}(x, \la_0) \equiv 0$. This contradiction yields (ii).
\end{proof}

In view of the asymptotics \eqref{asymptla}, we have $\la_{l,k} \ne \la_{r,k + 1}$ for sufficiently large $l$ and $r$. Therefore, Lemma~\ref{lem:N2} implies the following corollary.

\begin{cor} \label{cor:N}
For sufficiently large $|\la_0|$, $\la_0 \in \Lambda$, all the entries of $\mathcal N(\la_0)$ equal zero except $\mathcal N_{k+1,k}(\la_0)$, $k = \overline{1,n-1}$.
\end{cor}

Define \textit{the weight numbers}
$\beta_{l,k} := \mathcal N_{k+1,k}(\la_{l,k})$. It is worth considering $\beta_{l,k}$ only for sufficiently large $l$.
It follows from \eqref{defN} and \eqref{Mjk} that
$$
\beta_{l,k} = M_{k+1,k,\langle -1\rangle}(\la_{l,k}) = -\frac{\Delta_{k+1,k}(\la_{l,k})}{\tfrac{d}{d\la} \Delta_{k,k}(\la_{l,k})}.
$$
Consequently, Theorem~6.2 from \cite{Bond22-asympt} yields the asymptotics
\begin{equation} \label{asymptbe}
    \beta_{l,k} = l^{n-1 + p_{k+1,0} - p_{k,0}} (\beta^0_k + \varkappa_{l,k}^0), \quad \{ \varkappa_{l,k}^0 \} \in l_2, \quad k = \overline{1,n-1},
\end{equation}
where the constants $\be^0_k$ depend only on $n$, $k$, and $\{ p_{s,a} \}$.

\section{Main equation} \label{sec:main}

This section is devoted to the constructive solution of the auxiliary Problem~\ref{prob:sd}, that is, to the recovery of the Weyl solutions $\{ \Phi_k(x, \la) \}_{k = 1}^n$ from the spectral data $\{ \la_0, \mathcal N(\la_0) \}_{\la_0 \in \Lambda}$. We consider this problem 
for $\mathcal L = (F(x), U_0, U_1) \in W$ with an arbitrary $F \in \mathfrak F_n$. Thus, the results of this section can be applied to a wide class of differential expressions \eqref{defl} with associated matrix of $\mathfrak F_n$.

Along with $\mathcal L$, we consider another problem $\tilde {\mathcal L} = (\tilde F(x), \tilde U_0, \tilde U_1)$ of the same form but with different coefficients. Assume that $\tilde F \in \mathfrak F_n$, $p_{s,a} = \tilde p_{s,a}$, $s = \overline{1, n}$, $a = 0, 1$. 
The quasi-derivatives for $\tilde {\mathcal L}$ are defined by the matrix $\tilde F(x)$, so they are different from the quasi-derivatives of the problem $\mathcal L$.
The problem $\tilde {\mathcal L}^{\star}$ is defined similarly to $\mathcal L^{\star}$.
For simplicity, we assume that $\tilde {\mathcal L} \in W$. The case $\tilde {\mathcal L} \in W$ requires technical modifications (see Remark~\ref{rem:mult}). Denote $\mathcal I := \Lambda \cup \tilde \Lambda$.

In Subsection~\ref{sec:cont}, we reduce the studied problem to the infinite system \eqref{infphi} of linear equations with respect to some entries of $\phi_{\langle 0 \rangle}(x,\la_0)$, $\la_0 \in \mathcal I$. Our technique is based on the contour integration in the $\la$-plane and on the Residue theorem. In Subsection~\ref{sec:Banach}, the system \eqref{infphi} is transformed into the main equation \eqref{main} in the Banach space $m$ of infinite bounded sequences. The unique solvability of the main equation is proved. Finally, we arrive at the constructive Algorithm~\ref{alg:1} for finding $\{ \Phi_k(x, \la) \}_{k = 1}^n$ by the spectral data. This algorithm will be used in the next section for solving the inverse spectral problem.

\subsection{Contour integration} \label{sec:cont}

In order to formulate and prove the main lemma of this subsection (Lemma~\ref{lem:cont}), we first need some preliminaries. Introduce the notations
\begin{gather} \label{defD}
D(x, \mu, \la) := (\la - \mu)^{-1} [\Phi(x, \mu)]^{-1} \Phi(x, \la), \quad \tilde D(x, \mu, \la) := (\la - \mu)^{-1} [\tilde \Phi(x, \mu)]^{-1} \tilde \Phi(x, \la), \\ \label{defDa}
D_{\langle \al \rangle}(x, \la_0, \la) := [D(x, \mu, \la)]_{|\mu = \la_0}^{\langle \al \rangle},  \quad \al \in \mathbb Z.
\end{gather}
and similarly define $\tilde D_{\langle \al \rangle}(x, \la_0, \la)$.  % Analogously define $\phi^{\star}(x, \la)$, $\tilde \phi(x, \la)$, $\tilde \phi^{\star}(x, \la)$.

\begin{lem} \label{lem:D}
The following relations hold:
\begin{align} \label{DN1}
    & D_{\langle -1 \rangle}(x, \la_0, \la) = - \mathcal N(\la_0) D_{\langle 0 \rangle}(x, \la_0, \la), \\ \label{DN2}
    & [D(x, \mu, \la)]_{|\la = \la_0}^{\langle -1 \rangle} = [D(x, \mu, \la)]_{|\la = \la_0}^{\langle 0 \rangle} \mathcal N(\la_0), \\ \label{DN3}
    & [(\la - \la_0) I + \mathcal N(\la_0)] D_{\langle 0 \rangle}(x, \la_0, \la) = J_0^{-1} \langle [\phi^{\star}_{\langle 0 \rangle}(x, \la_0)]^T, \phi(x, \la) \rangle, \\ \label{Ddx}
    & D'(x, \mu, \la) = J_0^{-1} [\phi^{\star}(x, \mu)]^T \phi(x, \la).
\end{align}
\end{lem}

\begin{proof}
Using \eqref{PJP} and \eqref{defD}, we obtain
\begin{equation} \label{relD}
D(x, \mu, \la) = (\la - \mu)^{-1} J_0^{-1} [\Phi^{\star}(x, \mu)]^T J \Phi(x, \la).
\end{equation}
It follows from \eqref{relD} and \eqref{defDa} that
\begin{align} \label{Dm1}
    & D_{\langle -1 \rangle}(x, \la_0, \la) =   
    (\la - \la_0)^{-1} J_0^{-1} [\Phi^{\star}_{\langle -1 \rangle}(x, \la_0)]^T J \Phi(x, \la), \\ \label{D0}
    & D_{\langle 0 \rangle}(x, \la_0, \la) =
    (\la - \la_0)^{-1} J_0^{-1} [\Phi^{\star}_{\langle 0 \rangle}(x, \la_0)]^T J \Phi(x, \la) + (\la - \la_0)^{-2} J_0^{-1} [ \Phi^{\star}_{\langle -1 \rangle}(x, \la_0)]^T J \Phi(x, \la).
\end{align}
Using \eqref{Dm1},\eqref{D0} together with Lemma~\ref{lem:N1}, we derive \eqref{DN1}. The relation \eqref{DN2} is proved similarly.

It follows from \eqref{wrona1} that
\begin{equation} \label{wronphi}
[\Phi^{\star}(x, \mu)]^T J \Phi(x, \la) = \langle [\phi^{\star}(x, \mu)]^T, \phi(x, \la) \rangle.
\end{equation}
Using \eqref{Dm1}, \eqref{D0}, and \eqref{wronphi}, we obtain
$$
    (\la - \la_0) D_{\langle 0 \rangle}(x, \la_0, \la) = J_0^{-1} \langle [\phi_{\langle 0 \rangle}^{\star}(x, \mu)]^T, \phi(x, \la) \rangle + D_{\langle -1 \rangle}(x, \la_0, \la).
$$
Taking \eqref{DN1} into account, we arrive at \eqref{DN3}.

In order to prove \eqref{Ddx}, we combine \eqref{relD}, \eqref{wronphi}, and \eqref{wron2}:
$$
D'(x, \mu, \la) = (\la - \mu)^{-1} J_0^{-1} \frac{d}{dx} \langle [\phi^{\star}(x, \mu)]^T, \phi(x, \la) \rangle = J_0^{-1} [\phi^{\star}(x, \mu)]^T \phi(x, \la).
$$
\end{proof}

Put $\hat {\mathcal N}(\la_0) := \mathcal N(\la_0) - \tilde {\mathcal N}(\la_0)$. Below in this section, we suppose that $x \in [0, 1]$ is fixed.

\begin{lem} \label{lem:cont}
The following relations hold:
\begin{gather} \label{contphi}
    \phi(x, \la) = \tilde \phi(x, \la) + \sum_{\la_0 \in \mathcal I} \phi_{\langle 0 \rangle} (x, \la_0) \hat {\mathcal N}(\la_0) \tilde D_{\langle 0\rangle}(x, \la_0, \la), \\ \label{contD}
    D(x, \mu, \la) - \tilde D(x, \mu, \la) = \sum_{\la_0 \in \mathcal I} [D(x, \mu, \xi)]_{\xi = \la_0}^{\langle 0 \rangle} \hat {\mathcal N}(\la_0) \tilde D_{\langle 0 \rangle}(x, \la_0, \la),
\end{gather}
where the series converge in the sense 
$$
\sum_{\la_0 \in \mathcal I} = \lim_{R \to \iy} \sum_{\la_0 \in \mathcal I_R}, \quad \mathcal I_R := \{ \la \in \mathcal I \colon |\la| < R \},
$$
uniformly by $\la, \mu$ on compact sets of $(\mathbb C \setminus \mathcal I)$.
\end{lem}

\begin{proof} 
In this proof, a crucial role is played by the matrix of spectral mappings 
\begin{equation} \label{defP}
    \mathcal P(x, \la) = \Phi(x, \la) [\tilde \Phi(x, \la)]^{-1}.
\end{equation}
It follows from \eqref{PJP} and \eqref{defP} that
\begin{equation} \label{relP}
\mathcal P(x, \la) = \Phi(x, \la) J_0^{-1} [\tilde \Phi^{\star}(x,\la)]^T J.
\end{equation}

The proof consists of three steps.

\smallskip

\textsc{ Step 1. Regions and contours.} 
Choose a circle $\mathcal C_* := \{ \la \in \mathbb C \colon |\la| < \la_* \}$ of sufficiently large radius $\la_*$. 
Choose the $\sqrt[n]{\la}$ branch so that $\arg(\sqrt[n]{\la}) \in \left( -\tfrac{\pi}{2n}, \tfrac{3\pi}{2n} \right)$.
Then, it follows from the asymptotics \eqref{asymptla} that the roots $\rho_0 := \sqrt[n]{\la_0}$ of the eigenvalues $\la_0 \in (\mathcal I \setminus \mathcal C_*)$ lie in the two strips
\begin{equation} \label{defSj}
\mathcal S_j := \{ \rho \colon \mbox{Re}\, (\epsilon_j \rho) > 0, \, |\mbox{Im} (\epsilon_j \rho)| < c \}, \quad \epsilon_j := \exp(-2\pi \mathrm{i} j/n), \quad j = 0, 1,
\end{equation}
for an appropriate choice of the constant $c$. More precisely, $\sqrt[n]{\la_{l,k}} \in \mathcal S_0$ if $(n-k)$ is even and $\sqrt[n]{\la_{l,k}} \in \mathcal S_1$ otherwise. For $j = 0, 1$, denote by $\Xi_j$ the image of $\mathcal S_j$ in the $\la$-plane under the mapping $\la = \rho^n$. Put $\Xi := \Xi_0 \cup \Xi_1 \cup \mathcal C_*$. Clearly, $\mathcal I \subset \Xi$. 

Further, fix a sufficiently small $\de > 0$ and define the regions
$$
\mathcal S_{j,\de} := \{ \rho \in \mathcal S_j \colon \exists \rho_0 \in \mathcal S_j \cap \mathcal I \:\text{s.t.}\: |\rho - \rho_0| < \de \}, \quad j = 0, 1.
$$
For $j = 0, 1$, denote by $\Xi_{j,\de}$ the image of $\mathcal S_j$ in the $\la$-plane under the mapping $\la = \rho^n$. Put 
$$
\mathcal H_{\de} := \mathbb C \setminus (\Xi_{1,\de} \cup \Xi_{2,\de} \cup \mathcal C_*).
$$
Let $\la = \rho^n$, $\Theta(\rho) := \diag \{ 1, \rho, \ldots, \rho^{n-1} \}$.
It can be shown in the standard way (see, e.g., the relation (2.1.37) in \cite{Yur02} and the proof of Theorem~2 in \cite{Bond21}) that
\begin{equation} \label{asymptP}
\mathcal P(x, \la) = \Theta(\rho)(I + o(1))[\Theta(\rho)]^{-1}, \quad |\la| \to \iy,
\end{equation}
uniformly with respect to $\la \in \mathcal H_{\de}$.

For sufficiently large values of $R > 0$, define the regions (see Fig.~\ref{img:contours}):
$$
\Xi_R := \{ \la \in \Xi \colon |\la| < R \}, \quad \Xi_R^{\pm} := \{ \la \colon |\la| < R, \, \la \not\in \Xi, \, \pm \mbox{Im}\la > 0 \},
$$
and their boundaries $\ga_R := \partial \Xi_R$, $\ga_R^{\pm} := \partial \Xi_R^{\pm}$ with the counter-clockwise circuit.
Below we consider only such radii $R$ that $\ga_R \subset \mathcal H_{\de}$.

\begin{figure}[h!]
\centering
\begin{tikzpicture}[scale = 0.7]
\draw (1,1) arc(45:135:1.41);
\draw (1,-1) arc(-45:-135:1.41);
\draw (1,1) .. controls (2, 1.8) .. (5,3);
\draw (1,-1) .. controls (2,-1.8) .. (5,-3);
\draw (-1,1) .. controls (-2, 1.8) .. (-5,3);
\draw (-1,-1) .. controls (-2, -1.8) .. (-5,-3);
\draw (0, 0) circle (4);
\draw  (0,0) node{$\Xi_R$};
\draw (0,3) node{$\Xi_R^+$};
\draw (0,-3) node{$\Xi_R^-$};
\filldraw (3,0) circle (1pt);
\draw[dashed] (3,0) circle (0.7);
\filldraw (5.3,0) circle (1pt);
\draw[dashed] (5.3,0) circle (1);
\filldraw (-3,0) circle (1pt);
\draw[dashed] (-3,0) circle (0.7);
\filldraw (-5.3,0) circle (1pt);
\draw[dashed] (-5.3,0) circle (1);
\end{tikzpicture}
\caption{Contours}
\label{img:contours}
\end{figure}

\smallskip

\textsc{Step 2. Contour integration.}
In view of \eqref{relP}, the matrix function $\mathcal P(x, \la)$ is meromorpic in $\la$ with the poles $\mathcal I$. Hence, $\mathcal P(x, \la)$ is analytic in $\Xi_R^{\pm}$. 
Let $\mathcal P_1(x, \la)$ be the first row of $\mathcal P(x, \la)$.
The Cauchy formula implies
\begin{gather*}
\mathcal P_1(x, \la) - e_1^T = -\frac{1}{2\pi \mathrm{i}} \oint\limits_{\ga_R^{\pm}} \frac{\mathcal P_1(x, \xi) - e_1^T}{\la - \xi} \, d\xi, \quad \la \in \Xi_R^{\pm}, \\
\frac{\mathcal P(x, \la) - \mathcal P(x, \mu)}{\la - \mu} = -\frac{1}{2\pi \mathrm{i}} \oint\limits_{\ga_R^{\pm}} \frac{\mathcal P(x, \xi)}{(\la - \xi)(\xi - \mu)} \, d\xi, \quad \la, \mu \in \Xi_R^{\pm}.
\end{gather*}
Consequently,
\begin{gather} \label{smP1}
\mathcal P_1(x, \la) = e_1^T + \frac{1}{2 \pi \mathrm{i}} \oint\limits_{\ga_R} \frac{\mathcal P_1(x, \xi)}{\la - \xi}\, d\xi - \frac{1}{2 \pi \mathrm{i}} \oint\limits_{|\xi| = R} \frac{\mathcal P_1(x, \xi) - e_1^T}{\la - \xi}\, d\xi,  \\ \label{smP2}
\frac{\mathcal P(x, \la) - \mathcal P(x, \mu)}{\la - \mu} =  \frac{1}{2 \pi \mathrm{i}} \oint\limits_{\ga_R} \frac{\mathcal P(x, \xi)}{(\la - \xi)(\xi - \mu)} \, d\xi - \frac{1}{2 \pi \mathrm{i}} \oint\limits_{|\xi| = R} \frac{\mathcal P(x, \xi)}{(\la - \xi)(\xi - \mu)} \, d\xi. 
\end{gather}
Using \eqref{defP}, \eqref{defD}, \eqref{smP1}, and \eqref{smP2}, we derive
\begin{align} \label{smP3}
\phi(x, \la) = \mathcal P_1(x, \la) \tilde \Phi(x, \la) & = \tilde \phi(x, \la) + \frac{1}{2 \pi \mathrm{i}} \oint\limits_{\ga_R} \frac{\mathcal P_1(x, \xi) \tilde \Phi(x, \la)}{\la - \xi}\, d\xi
+ \eps_R^1(x, \la), \\ \nonumber
D(x, \mu, \la) - \tilde D(x, \mu, \la) & = \frac{[\Phi(x, \mu)]^{-1} (\mathcal P(x, \la) - \mathcal P(x, \mu)) \tilde \Phi(x, \la)}{\la - \mu} \\ \nonumber & = 
\frac{1}{2 \pi \mathrm{i}} \oint\limits_{\ga_R} \frac{[\Phi(x, \mu)]^{-1} \Phi(x, \xi)}{\xi - \mu} \frac{[\tilde \Phi(x, \xi)]^{-1} \tilde \Phi(x, \la)}{\la - \xi} \, d\xi + \eps_R^2(x, \mu, \la) \\ \label{smP4} & =
\frac{1}{2 \pi \mathrm{i}} \oint\limits_{\ga_R} D(x, \mu, \xi) \tilde D(x, \xi, \la) \, d\xi + \eps_R^2(x, \mu, \la),
\end{align}
where
\begin{align*}
\eps_R^1(x, \la) := & - \frac{1}{2 \pi \mathrm{i}} \oint\limits_{|\xi| = R} \frac{(\mathcal P_1(x, \xi) - e_1^T) \tilde \Phi(x, \la)}{\la - \xi}\, d\xi, \\
\eps_R^2(x, \mu, \la) := & - \frac{1}{2 \pi \mathrm{i}} \oint\limits_{|\xi| = R} \frac{[\Phi(x, \mu)]^{-1} \mathcal P(x, \xi) \tilde \Phi(x, \la)}{(\la - \xi) (\xi - \mu)} \, d\xi.
\end{align*}
It follows from \eqref{asymptP} that 
\begin{equation} \label{limeps}
\lim_{\substack{R \to \iy \\ \ga_R \subset \mathcal H_{\de}}} \eps_R^1(x, \la) = 0, \quad \lim_{\substack{R \to \iy \\ \ga_R \subset \mathcal H_{\de}}} \eps_R^2(x, \mu, \la) = 0.
\end{equation}

\smallskip

\textsc{Step 3. Residues}.
Using the first row of \eqref{relP}:
$$
\mathcal P_1(x, \la) = \phi(x, \la) J_0^{-1} [\tilde \Phi^{\star}(x,\la)]^T J
$$
and the Residue theorem, we obtain
\begin{equation} \label{PRes}
\frac{1}{2 \pi \mathrm{i}} \oint\limits_{\ga_R} \frac{\mathcal P_1(x, \xi) \tilde \Phi(x, \la)}{\la - \xi}\, d\xi = \sum_{\la_0 \in \mathcal I_R} \Res_{\xi = \la_0} \phi(x, \xi) \tilde D(x, \xi, \la). 
\end{equation}

Using \eqref{smP3}, \eqref{limeps}, and \eqref{PRes}, we get
\begin{equation} \label{smphi1}
\phi(x, \la) = \tilde \phi(x, \la) + \sum_{\la_0 \in \mathcal I} (\phi_{\langle -1 \rangle}(x, \la_0) \tilde D_{\langle 0 \rangle}(x, \la_0, \la) + \phi_{\langle 0 \rangle}(x, \la_0)  \tilde D_{\langle -1 \rangle}(x, \la_0, \la)).
\end{equation}
It follows from \eqref{relNPhi} that
\begin{equation} \label{smphi2}
\phi_{\langle -1 \rangle}(x, \la_0) = \phi_{\langle 0 \rangle}(x, \la_0) \mathcal N(\la_0).
\end{equation}
Substituting \eqref{DN1} for $\tilde D_{\langle -1\rangle}(x, \la_0, \la)$ and \eqref{smphi2} into \eqref{smphi1}, we derive the relation \eqref{contphi}.

It remains to prove \eqref{contD}. Using Lemma~\ref{lem:D}, we derive
\begin{align} \nonumber
    \Res_{\xi = \la_0} D(x, \mu, \xi) \tilde D(x, \xi, \la) & = 
    [D(x, \mu, \xi)]_{|\xi = \la_0}^{\langle -1 \rangle} \tilde D_{\langle 0 \rangle}(x, \la_0, \la) + [D(x, \mu, \xi)]_{|\xi = \la_0}^{\langle 0 \rangle} \tilde D_{\langle -1 \rangle}(x, \la_0, \la) \\ \label{ResD} & =
    [D(x, \mu, \xi)]_{|\xi = \la_0}^{\langle 0 \rangle} \hat {\mathcal N}(\la_0) \tilde D_{\langle 0 \rangle}(x, \la_0, \la).
\end{align}
Combining \eqref{smP4}, \eqref{limeps}, \eqref{ResD} all together and applying the Residue theorem, we arrive at \eqref{contD}.

Now \eqref{contphi} and \eqref{contD} are proved only for $\la, \mu \in (\mathbb C \setminus \Xi)$. Using analytic continuation, we conclude that these relations hold for $\la, \mu \in (\mathbb C \setminus \mathcal I)$.
\end{proof}

Our next goal is to obtain an infinite system of linear equations with respect to some entries of $\phi_{\langle 0 \rangle}(\la_0)$, $\la_0 \in \mathcal I$.
Introduce the ordered set
$$
V := \{ (l,k,\eps) \colon l \ge 1, \, k \in \{ 1, \ldots, n-1 \}, \, \eps \in \{ 0, 1 \}.
$$
For $v = (l,k,\eps)$, $v_0 = (l_0,k_0,\eps_0)$, $v,v_0 \in V$, we mean that $v < v_0$ if $l < l_0$ or $(l = l_0 \: \text{and} \: k < k_0)$ or $(l = l_0, \, k= k_0 \: \text{and} \: \eps < \eps_0)$.
Denote
\begin{gather} \label{lalk}
\la_{l,k,0} := \la_{l,k}, \quad \la_{l,k,1} := \tilde \la_{l,k}, \quad \mathcal N_0(\la_0) := \mathcal N(\la_0), \quad \mathcal N_1(\la_0) := \tilde {\mathcal N}(\la_0), \\ \label{defphilk}
\vv_{l,k,\eps}(x) := \Phi_{k+1,\langle 0 \rangle}(x, \la_{l,k,\eps}), \quad  
\tilde \vv_{l,k,\eps}(x) := \tilde \Phi_{k+1,\langle 0 \rangle}(x, \la_{l,k,\eps}),  \\ \label{defPlk}
\tilde P_{l,k,\eps}(x, \la) := e_{k+1}^T \mathcal N_{\eps}(\la_{l,k,\eps}) \tilde D_{\langle 0 \rangle}(x, \la_{l,k,\eps}, \la), \\
\label{defG}
\tilde G_{(l,k,\eps), (l_0,k_0,\eps_0)} (x) := [\tilde P_{l,k,\eps}(x,\la)]_{\la = \la_{l_0,k_0,\eps_0}}^{\langle 0 \rangle} e_{k_0+1},
\end{gather}
and similarly define $P_{l,k,\eps}(x, \la)$, $G_{(l,k,\eps),(l_0,k_0,\eps_0)}(x)$. Using these notations, we obtain the following corollary of Lemma~\ref{lem:cont}.

\begin{cor} \label{cor:inf}
The following relations hold:
\begin{gather} \label{findphi}
\phi(x, \la) = \tilde \phi(x, \la) + \sum_{(l,k,\eps) \in V} (-1)^{\eps} \vv_{l,k,\eps}(x) \tilde P_{l,k,\eps}(x, \la), \\ \label{infphi}
\vv_{l_0, k_0,\eps_0}(x) = \tilde \vv_{l_0,k_0,\eps_0}(x) + \sum_{(l,k,\eps) \in V}(-1)^{\eps} \vv_{l,k,\eps}(x) \tilde G_{(l,k,\eps), (l_0, k_0,\eps_0)}(x), \\ \label{infG}
G_{(l_0,k_0,\eps_0), (l_1,k_1,\eps_1)}(x) - \tilde G_{(l_0,k_0,\eps_0), (l_1,k_1,\eps_1)}(x) = \sum_{(l,k,\eps) \in V} (-1)^{\eps} G_{(l_0,k_0,\eps_0),(l,k,\eps)}(x) \tilde G_{(l, k,\eps), (l_1, k_1,\eps_1)}(x),
\end{gather}
where $x \in [0, 1]$, $(l_0,k_0,\eps_0), (l_1,k_1,\eps_1) \in V$.
\end{cor}

\begin{proof}
Taking Lemma~\ref{lem:N2} on the structure of $\mathcal N(\la_0)$ and $\tilde{\mathcal N}(\la_0)$ into account, we rewrite \eqref{contphi} in the form
\begin{equation*} 
\phi(x, \la) = \tilde \phi(x, \la) + \sum_{(l,k,\eps) \in V} (-1)^{\eps} \Phi_{k+1, \langle 0 \rangle}(x, \la_{l, k,\eps}) e_{k+1}^T \mathcal N_\eps(\la_{l,k,\eps}) \tilde D_{\langle 0 \rangle }(x, \la_{l,k,\eps}, \la). 
\end{equation*}
Using \eqref{defphilk} and \eqref{defPlk}, we arrive at \eqref{findphi}. Taking the $(k_0+1)$-th entry in the relation \eqref{findphi}, putting $\la = \la_{l_0,k_0,\eps_0}$ and using \eqref{defphilk}, \eqref{defG}, we readily obtain \eqref{infphi}.

Analogously, we represent \eqref{contD} as follows:
$$
D(x, \mu, \la) - \tilde D(x, \mu, \la) = \sum_{(l,k,\eps) \in V} (-1)^{\eps} [D(x, \mu,\xi)]_{\xi = \la_{l,k,\eps}}^{\langle 0 \rangle} e_{k+1} e_{k+1}^T \mathcal N_{\eps}(\la_{l,k,\eps}) \tilde D_{\langle 0 \rangle}(x, \la_{l,k,\eps}, \la).
$$
Passing from $D(x, \mu,\la)$ and $\tilde D(x, \mu, \la)$ to $P_{l_0,k_0,\eps_0}(x, \la)$ and $\tilde P_{l_0,k_0,\eps_0}(x, \la)$, respectively, we derive
$$
P_{l_0,k_0,\eps_0}(x, \la) - \tilde P_{l_0,k_0,\eps_0}(x, \la) = \sum_{(l,k,\eps) \in V} (-1)^{\eps} [P_{l_0,k_0,\eps_0}(x, \xi)]_{\xi = \la_{l,k,\eps}}^{\langle 0 \rangle} e_{k+1} \tilde P_{l,k,\eps}(x, \la).
$$
Using \eqref{defG} and the analogous relation for $G_{(l,k,\eps), (l_0,k_0,\eps_0)}(x)$, we finally arrive \eqref{infG}.
\end{proof}

The relations \eqref{infphi} can be considered as an infinite linear system with respect to $\vv_{l,k,\eps}(x)$, $(l,k,\eps) \in V$. However, it is inconvenient to use \eqref{infphi} as the main equation system for the inverse problem, because the series in \eqref{infphi} converges only ``with brackets'':
$$
\sum_{(l,k,\eps) \in V} = \sum_{(l,k)} \left( \sum_{\eps = 0, 1} (\dots)\right).
$$
Therefore, in the next section, we transform the system \eqref{infphi} to a linear equation in a suitable Banach space. The relation \eqref{infG} will be used to prove the unique solvability of the main equation.

\begin{remark} \label{rem:mult}
If $\tilde {\mathcal L} \not\in W$, that is, the poles of $\tilde M(\la)$ are not necessarily simple, then this influences on the calculation of the residues in \eqref{PRes}.
Consequently,
we obtain the following relation instead of \eqref{contphi}:
\begin{align} \nonumber
\phi(x, \la) = \tilde \phi(x, \la) + \sum_{\la_0 \in \mathcal I} \biggl[ & \phi_{\langle 0 \rangle}(x, \la_0) (\mathcal N(\la_0) \tilde D_{\langle 0 \rangle}(x, \la_0, \la) + \tilde D_{\langle -1 \rangle}(x, \la_0, \la)) \\ \label{contphi1} & + \sum_{k = 1}^{m_{\la_0}-1} \phi_{\langle k \rangle}(x, \la_0) \tilde D_{\langle -(k+1) \rangle}(x, \la_0, \la)\biggr],
\end{align}
where $m_{\la_0}$ is the multiplicity of $\la_0 \in \tilde \Lambda$. Using \eqref{contphi1}, one can derive an infinite system analogous to \eqref{infphi}, containing not only entries of the vectors $\phi_{\langle 0 \rangle}(x, \la_0)$ but also of $\phi_{\langle k \rangle}(x, \la_0)$ for $k = \overline{1,m_{\la_0}-1}$.
\end{remark}

\subsection{Linear equation in a Banach space} \label{sec:Banach}

Define the numbers $\{ \xi_l \}$, which characterize ``the difference'' of the two spectral data sets $\{ \la_0, \mathcal N(\la_0) \}_{\la_0 \in \Lambda}$ and $\{ \tilde \la_0, \tilde {\mathcal N}(\tilde \la_0) \}_{\tilde \la_0 \in \tilde \Lambda}$:
\begin{equation} \label{defxi}
\xi_l := \sum_{k = 1}^{n-1} \left( |\la_{l,k} - \tilde \la_{l,k}| + \sum_{j = k+1}^n |\mathcal N_{j,k}(\la_{l,k}) - \tilde{\mathcal N}_{j,k}(\tilde \la_{l,k})| l^{p_{k,0} - p_{k+1,0}}\right) l^{1-n}, \quad l \ge 1.
\end{equation}

Taking Corollary~\ref{cor:N} into account, we reduce \eqref{defxi} to the following form for all sufficiently large values of $l$:
\begin{equation} \label{relxi}
    \xi_l = \sum_{k = 1}^{n-1} \left( |\la_{l,k} - \tilde \la_{l,k}| + |\beta_{l,k} - \tilde \beta_{l,k}| l^{p_{k,0} - p_{k+1,0}}\right) l^{1-n}.
\end{equation}

Relation~\eqref{relxi} together with the asymptotics \eqref{asymptla} and \eqref{asymptbe} imply $\{ \xi_l \} \in l_2$.

\begin{lem} \label{lem:est}
The following estimates hold for $(l,k,\eps), (l_0,k_0,\eps_0) \in V$:
\begin{gather*}
    |\vv_{l,k,\eps}(x)| \le C w_{l,k}(x), \quad
    |\vv_{l,k,0}(x) - \vv_{l,k,1}(x)| \le C w_{l,k}(x) \xi_l, \\
    |G_{(l,k,\eps),(l_0,k_0,\eps_0)}(x)| \le \frac{C}{|l - l_0| + 1} \cdot \frac{w_{l_0,k_0}(x)}{w_{l,k}(x)}, \\
    |G_{(l,k,0),(l_0,k_0,\eps_0)}(x) - G_{(l,k,1),(l_0,k_0,\eps_0)}(x)| \le \frac{C \xi_l}{|l - l_0| + 1} \cdot \frac{w_{l_0,k_0}(x)}{w_{l,k}(x)}, \\
    |G_{(l,k,\eps),(l_0,k_0,0)}(x) - G_{(l,k,\eps),(l_0,k_0,1)}(x)| \le \frac{C \xi_{l_0}}{|l - l_0| + 1} \cdot \frac{w_{l_0,k_0}(x)}{w_{l,k}(x)}, \\
    |G_{(l,k,0),(l_0,k_0,0)}(x) - G_{(l,k,0),(l_0,k_0,1)}(x)
    - G_{(l,k,1),(l_0,k_0,0)}(x) + G_{(l,k,1),(l_0,k_0,1)}(x) \le \\ \frac{C \xi_l \xi_{l_0}}{|l-l_0| + 1} \cdot \frac{w_{l_0,k_0}(x)}{w_{l,k}(x)},
\end{gather*}
where
$$
w_{l,k}(x) := l^{-p_{k+1,0}} \exp(-xl \cot (k\pi/n)),
$$
and the constant $C$ does not depend on $x, l,\eps,k, l_0,\eps_0,k_0$.
\end{lem}

The proof of Lemma~\ref{lem:est} repeats the technique of \cite[Section 2.3.3]{Yur02}, so we omit it. 
The similar estimates are valid for $\tilde \vv_{l,k,\eps}(x)$ and $\tilde G_{(l_0,k_0,\eps_0),(l,k,\eps)}(x)$.

Put $\theta_l := \xi_l^{-1}$ if $\xi_l \ne 0$ and $\theta_l = 0$ otherwise. Introduce the notations
\begin{equation} \label{defpsi}
\begin{bmatrix}
\psi_{l,k,0}(x) \\ \psi_{l,k,1}(x)
\end{bmatrix} := 
w_{l,k}^{-1}(x)
\begin{bmatrix}
\theta_l & -\theta_l \\ 0 & 1
\end{bmatrix}
\begin{bmatrix}
\vv_{l,k,0}(x) \\ \vv_{l,k,1}(x)
\end{bmatrix}, 
\end{equation}
\begin{multline} \label{defR}
\begin{bmatrix}
R_{(l_0,k_0,0),(l,k,0)}(x) & R_{(l_0,k_0,0),(l,k,1)}(x) \\
R_{(l_0,k_0,1),(l,k,0)}(x) & R_{(l_0,k_0,1),(l,k,1)}(x)
\end{bmatrix} := \\ 
\frac{w_{l,k}(x)}{w_{l_0,k_0}(x)}
\begin{bmatrix}
\theta_{l_0} & -\theta_{l_0} \\ 0 & 1
\end{bmatrix}
\begin{bmatrix}
G_{(l,k,0),(l_0,k_0,0)}(x) & G_{(l,k,1),(l_0,k_0,0)}(x) \\
G_{(l,k,0),(l_0,k_0,1)}(x) & G_{(l,k,1),(l_0,k_0,1)}(x)
\end{bmatrix}
\begin{bmatrix}
\xi_{l} & 1 \\ 0 & -1
\end{bmatrix}.
\end{multline}
For brevity, put $\psi_v(x) := \psi_{l,k,\eps}(x)$, $R_{v_0,v}(x) := R_{(l_0,k_0,\eps_0),(l,k,\eps)}(x)$, $v = (l,k,\eps)$, $v_0 = (l_0,k_0,\eps_0)$, $v,v_0 \in V$. The functions $\tilde \psi_v(x)$ and $\tilde R_{v_0,v}(x)$ are defined analogously.

Using \eqref{infphi}, \eqref{infG}, and the above notations, we obtain
\begin{gather} \label{sumpsi}
    \psi_{v_0}(x) = \tilde \psi_{v_0}(x) + \sum_{v \in V} \tilde R_{v_0,v}(x) \psi_v(x), \quad v_0 \in V, \\ \label{sumR}
    R_{v_1,v_0}(x) - \tilde R_{v_1,v_0}(x) = \sum_{v \in V} \tilde R_{v_1,v}(x) R_{v,v_0}(x), \quad v_1, v_0 \in V.
\end{gather}

Lemma~\ref{lem:est} yields the estimates
\begin{equation} \label{estpsiR}
|\psi_v(x)| \le C, \quad
|R_{v_0,v}(x)| \le \frac{C\xi_l}{|l-l_0| + 1}, \quad v,v_0 \in V,
\end{equation}
and the similar estimates for $\tilde \psi_v(x)$, $\tilde R_{v_0,v}(x)$. Consequently, the Cauchy-Bunyakovsky-Schwarz inequality
\begin{equation} \label{CBS}
\sum_{l} \frac{\xi_l}{|l-l_0| + 1} \le \left(\sum_{l} \xi_l^2 \right)^{1/2} \left(\sum_l \frac{1}{(|l-l_0| + 1)^2}\right)^{1/2} < \iy,
\end{equation}
implies the absolute convergence of the series in \eqref{sumpsi} and \eqref{sumR}.

Consider the Banach space $m$ of bounded infinite sequences $\al = [\al_v]_{v \in V}$ with the norm $\| \al \|_m = \sum\limits_{v \in V} |\al_v|$. Obviously, $\psi(x), \tilde \psi(x) \in m$ for each fixed $x \in [0,1]$. Define the linear operator $R(x) = [R_{v_0,v}(x)]_{v_0, v \in V}$ acting on an element $\al = [\al_v]_{v \in V} \in m$ by the following rule:
\begin{equation} \label{Ral}
[R(x) \al]_{v_0} = \sum_{v \in V} R_{v_0,v}(x) \al_v, \quad v_0 \in V.
\end{equation}
The operator $\tilde R(x) = [\tilde R_{v_0,v}(x)]_{v_0,v \in V}$ is defined similarly. It follows from \eqref{estpsiR} and \eqref{CBS} that the operators $R(x)$, $\tilde R(x)$ are bounded from $m$ to $m$ for each fixed $x \in [0,1]$. Denote by $\mathbf{I}$ the unit operator in $m$.

Using the introduced notations, we obtain the following theorem on the main equation and its unique solvability.

\begin{thm} \label{thm:main}
For each fixed $x\in[0,1]$, the linear operator $R(x)$ is compact in $m$ and can be approximated by finite-rank operators: $R(x) = \lim\limits_{N \to \iy} R^N(x)$. The same properties are valid for $\tilde R(x)$. Furthermore, the following relation holds
\begin{equation} \label{main}
    (\mathbf{I} - \tilde R(x)) \psi(x) = \tilde \psi(x), \quad x \in [0,1],
\end{equation}
which is called \textit{the main equation} of the inverse problem. The operator $(\mathbf{I} + \tilde R(x))$ has a bounded inverse of form
\begin{equation} \label{inv}
(\mathbf{I} - \tilde R(x))^{-1} = \mathbf{I} + R(x).
\end{equation}
Thus, the main equation~\eqref{main} is uniquely solvable in $m$ for each fixed $x \in [0,1]$.
\end{thm}

\begin{proof}
For $N \in \mathbb N$, define the index set 
$V^N := \{ v = (l,k,\eps) \in V \colon l \le N \}$
and the finite-rank operator $R^N(x)$:
\begin{equation} \label{RNal}
[R^N(x) \al]_{v_0} = \sum_{v \in V^N} R_{v_0,v}(x) \al_v.
\end{equation}
Using \eqref{estpsiR}--\eqref{RNal}, we show that
$$
\| R(x) - R^N(x) \|_{m \to m} = \sup_{v_0 \in V} \sum_{v \in (V\setminus V^N)} |R_{v_0,v}(x)| \le \sup_{l_0}\sum_{l \ge N} \frac{C\xi_l}{|l-l_0|+1} \to 0, \quad N \to \iy.
$$
Hence, the operator $R(x)$ is compact.

According to our notations, the relations \eqref{sumpsi} and \eqref{sumR} take the form \eqref{main} and 
$$
R(x) - \tilde R(x) = \tilde R(x) R(x),
$$
respectively. The latter relation implies \eqref{inv}, which completes the proof.
\end{proof}

Thus, we arrive at the following algorithm for solving Problem~\ref{prob:sd}.

\begin{alg} \label{alg:1}
Suppose that the spectral data $\{ \la_0, \mathcal N(\la_0) \}_{\la_0 \in \Lambda}$ of the problem $\mathcal L \in W$ are given. We have to find the Weyl solutions $\{ \Phi_k(x,\la) \}_{k = 1}^n$.

\begin{enumerate}
    \item Choose an arbitrary model problem $\tilde {\mathcal L} \in W$ with $\tilde p_{s,a} = p_{s,a}$, $s = \overline{1,n}$, $a = 0, 1$. In particular, one can take $\tilde F(x) = [\de_{k+1,j}]_{k,j = 1}^n$, $\tilde U_a = [\de_{j,p_{s,a}+1}]_{s,j = 1}^n$.
    \item For the problem $\tilde {\mathcal L}$, find the matrix function $\tilde \Phi(x, \la)$ and then $\tilde D(x, \mu,\la)$ by \eqref{defD}.
    \item Using $\tilde \Phi(x, \la)$, $\tilde D(x, \mu, \la)$, the spectral data $\{ \la_0, \mathcal N(\la_0)\}_{\la_0 \in \Lambda}$, $\{ \tilde \la_0, \tilde{\mathcal N}(\tilde \la_0)\}_{\tilde \la_0 \in \tilde \Lambda}$, and the notations \eqref{lalk}, find $\tilde \vv_{l,k,\eps}(x)$, $\tilde P_{l,k,\eps}(x,\la)$, and $\tilde G_{(l,k,\eps),(l_0,k_0,\eps_0)}$ for $(l,k,\eps), (l_0,k_0,\eps_0) \in V$ via \eqref{defphilk}, \eqref{defPlk}, and \eqref{defG}, respectively.
    \item Construct the infinite sequence $\tilde \psi(x)$ and the operator $\tilde R(x)$ by using \eqref{defpsi} and \eqref{defR} (with tilde), respectively.
    \item Find $\psi(x)$ by solving the main equation~\eqref{main}.
    \item Find $\{ \vv_{l,k,\eps}(x) \}_{(l,k,\eps) \in V}$ from \eqref{defpsi}:
    $$
        \begin{bmatrix}
            \vv_{l,k,0}(x) \\ \vv_{l,k,1}(x)
        \end{bmatrix} = w_{l,k}(x)
        \begin{bmatrix}
            \xi_l & 1 \\ 0 & 1
        \end{bmatrix}
        \begin{bmatrix}
            \psi_{l,k,0}(x) \\ \psi_{l,k,1}(x)
        \end{bmatrix}
    $$
    \item Construct $\phi(x,\la) = [\Phi_k(x, \la)]_{k = 1}^n$ by \eqref{findphi}.
\end{enumerate}
\end{alg}

\section{Reconstruction formulas} \label{sec:rec}

In this section, we use the solution $\psi(x)$ of the main equation \eqref{main} to obtain the solution of Problem~\ref{prob:sd-coef} for some classes of differential operators. We derive the reconstruction formulas in the form of series for the coefficients $\{ \tau_{\nu} \}_{\nu = 0}^{n-2}$ of the differential expression \eqref{defl}.

In Subsection~\ref{sec:gen}, the general approach to obtaining reconstruction formulas is described. However, for certain classes of the coefficients $\{ \tau_{\nu} \}_{\nu = 0}^{n-2}$, the convergence of the obtained series has to be studied in the corresponding spaces. Therefore, in Subsection~\ref{sec:series}, we prove an auxiliary lemma on the series convergence. In Subsections~\ref{sec:3}--\ref{sec:evenW}, we study the three classes of operators:

\medskip

(i) $n = 3$, $\tau_0 \in W_2^{-1}(0,1)$, $\tau_1 \in L_2(0,1)$;

\smallskip

(ii) $n$ is even, $\tau_{\nu} \in L_2(0,1)$, $\nu = \overline{0,n-2}$;

\smallskip

(iii) $n$ is even, $\tau_{\nu} \in W_2^{-1}(0,1)$, $\nu = \overline{0,n-2}$.

\medskip

For each case, we provide the uniqueness theorem of the inverse problem solution in an appropriate statement, obtain reconstruction formulas and prove the convergence of the series, and so get constructive algorithms for solving Problem~\ref{prob:sd-coef}. For the cases (ii) and (iii), we recover the coefficients $\tau_{n-2}$, $\tau_{n-3}$, \ldots, $\tau_1$, $\tau_0$ one-by-one in order to achieve the convergence estimates for the corresponding series. The even order in (ii) and (iii) is considered for definiteness. The similar ideas can be applied to the odd-order differential operators. For simplicity, in all the three cases, we choose such boundary conditions that their coefficients cannot be uniquely recovered from the spectral data and so do not consider their reconstruction. However, for other types of boundary conditions, the recovery of their coefficients also can be studied similarly to the regular case (see Lemma~2.3.7 in \cite{Yur02}).

Let us introduce some notations used throughout this section.
Note that the collection $\{ \la_{l,k,\eps} \}_{(l,k,\eps) \in V}$ may contain multiple eigenvalues for a fixed $\eps \in \{ 0, 1\}$: $\la_{l,k,\eps} = \la_{l_0,k_0,\eps}$, $(l,k) \ne (l_0,k_0)$. In order to exclude such values, we define the set
$$
V' := \{ (l,k,\eps) \in V \colon \not\exists (l_0,k_0,\eps) \in V \: \text{s.t.} \: (l_0,k_0) < (l,k) \: \text{and} \: \la_{l_0,k_0,\eps} = \la_{l,k,\eps} \}.
$$

In this section, we use the following notations for an index $v = (l,k,\eps) \in V'$:
\begin{gather} \label{defPv}
\la_v := \la_{l,k,\eps}, \quad \phi_v(x) := \phi_{\langle 0\rangle}(x, \la_v), \quad
\tilde P_v(x, \la) := (-1)^{\eps} \mathcal N_{\eps}(\la_v) \tilde D_{\langle 0 \rangle}(x, \la_v, \la), \\ \label{defcg}
c_v := (-1)^{\eps} \mathcal N_{\eps}(\la_v) J_0^{-1}, \quad \tilde g_v(x) := [\tilde \phi_{\langle 0\rangle}^{\star}(x,\la_v)]^T.
\end{gather}
Additionally, define the scalar functions
\begin{equation} \label{defeta}
\tilde \eta_{l,k,\eps}(x) := (-1)^{\eps} e_{k+1}^T \mathcal N_{\eps}(\la_{l,k,\eps})J_0^{-1} [\tilde \phi^{\star}_{\langle 0 \rangle}(x, \la_{l,k,\eps})]^T, \quad v \in V.
\end{equation}

\subsection{General approach} \label{sec:gen}

In terms of the notations \eqref{defPv}, the relation \eqref{contphi} can be rewritten as
$$
\phi(x, \la) = \tilde \phi(x, \la) + \sum_{v \in V'} \phi_v(x) \tilde P_v(x, \la). 
$$
Formal calculations show that
$$
\ell_n(\phi(x, \la)) = \ell_n(\tilde \phi(x, \la)) + \sum_{v \in V'} \ell_n(\phi_v(x) \tilde P_v(x, \la)).
$$
Recall that
$$
\ell_n(\phi(x, \la)) = \la \phi(x, \la), \quad 
\tilde \ell_n(\tilde \phi(x, \la)) = \la \tilde \phi(x, \la),
$$
and, by virtue of \eqref{lnphi},
$$
\ell_n(\phi_v(x)) = \la_v \phi_v(x) + \phi_v(x) \mathcal N_0(\la_v).
$$
Define $\hat \ell_n(y) := \ell_n(y) - \tilde \ell_n(y)$. Consequently,
\begin{multline} \label{sml4}
\la (\phi(x, \la) - \tilde \phi(x, \la)) - \sum_{v \in V'} \ell_n(\phi_v(x)) \tilde P_v(x, \la) = 
\sum_{v \in V'} \phi_v(x) [(\la - \la_v) I - \mathcal N_0(\la_v)] \tilde P_v(x, \la) \\ =
\hat \ell_n(\tilde \phi(x, \la)) + \sum_{v \in V'} \ell_n(\phi_v(x) \tilde P_v(x, \la)) - \sum_{v \in V'} \ell_n(\phi_v(x)) \tilde P_v(x, \la).
\end{multline}

Using \eqref{defPv} and \eqref{DN3}, we derive
\begin{align*}
[(\la - \la_v) - \mathcal N_0(\la_v)] \tilde P_v(x,\la) = & (-1)^{\eps} \mathcal N_{\eps}(\la_v) J_0^{-1} \langle [\tilde \phi_v^{\star}(x)]^T, \tilde \phi(x,\la)\rangle \\ & + (-1)^{\eps+1}[\mathcal N_{\eps}(\la_v) \mathcal N_1(\la_v) + \mathcal N_0(\la_v) \mathcal N_{\eps}(\la_v)] \tilde D_{\langle 0 \rangle}(x,\la_0,\la).
\end{align*}
The summation yields
\begin{equation} \label{sml2}
\sum_{v \in V'} \phi_v(x) [(\la - \la_v) I - \mathcal N_0(\la_v)] \tilde P_v(x, \la) = \sum_{v \in V'} \phi_v(x) c_v \langle \tilde g_v(x), \tilde \phi(x,\la)\rangle,
\end{equation}
where $c_v$ and $\tilde g_v(x)$ are defined by \eqref{defcg}.
Combining \eqref{sml4} and \eqref{sml2} together, we obtain
\begin{equation} \label{sml5}
\sum_{v \in V'} \phi_v(x) c_v \langle \tilde g_v(x), \tilde \phi(x,\la)\rangle =
\hat \ell_n(\tilde \phi(x, \la)) + \sum_{v \in V'} \ell_n(\phi_v(x) \tilde P_v(x, \la)) - \sum_{v \in V'} \ell_n(\phi_v(x)) \tilde P_v(x, \la).
\end{equation}

Suppose that the differential expression $y^{[n]} = \ell_n(y)$ has the form \eqref{defl}. Then, $\ell_n(y)$ can be formally represented as
\begin{equation} \label{deflp}
\ell_n(y) = y^{(n)} + \sum_{s = 0}^{n-2} p_s(x) y^{(s)},
\end{equation}
where
\begin{equation} \label{defps}
p_s = \sum_{k = \lceil s/2\rceil}^{\min \{s, \lfloor n/2\rfloor - 1\}} C_k^{s-k} [\tau_{2k}^{(2k-s)} + \tau_{2k+1}^{(2k-s+1)}] + \sum_{k = \lceil (s-1)/2\rceil}^{\min \{ s,\lfloor (n-1)/2 \rfloor\}-1} 2 C_k^{s-k-1} \tau_{2k+1}^{(2k+1-s)}.
\end{equation}
(We assume that $\tau_{n-1}(x) \equiv 0$). Suppose that $\tilde \ell_n(y)$ has the form similar to \eqref{deflp} with the coefficients $\tilde p_s(x)$, so
\begin{equation} \label{deflh}
\hat \ell_n(y) := \sum_{s = 0}^{n-2} \hat p_s(x) y^{(s)}, \quad \hat p_s := p_s - \tilde p_s.
\end{equation}

Using \eqref{deflp}, we derive
\begin{equation} \label{sml1}
\ell_n(\phi_v \tilde P_v) = \ell_n(\phi_v) \tilde P_v + \sum_{k = 1}^n C_n^k \sum_{v \in V'} \phi_v^{(n-k)} \tilde P_v^{(k)} + \sum_{k = 1}^{n-2} p_k \sum_{r = 1}^k C_k^r \sum_{v \in V'} \phi_v^{(k-r)} \tilde P_v^{(r)}.
\end{equation}
The relations \eqref{defPv} and \eqref{Ddx} imply
\begin{equation} \label{Pvdx}
\tilde P_v'(x, \la) = c_v \tilde g_v(x) \tilde \phi(x,\la).
\end{equation}
Substituting \eqref{Pvdx} into \eqref{sml1} and grouping the terms at $\tilde \phi^{(s)}(x, \la)$, we obtain
\begin{equation} \label{sml3}
\ell_n(\phi_v \tilde P_v) - \ell_n(\phi_v) \tilde P_v = \sum_{s = 0}^{n-1} t_{n,s} \tilde \phi^{(s)} + \sum_{s = 0}^{n-3} \sum_{k = s+1}^{n-2} p_k t_{k,s} \tilde \phi^{(s)},
\end{equation}
where
\begin{equation} \label{deftT}
t_{k,s}(x) := \sum_{r = s}^{k-1} C_k^{r+1} C_r^s T_{k-r-1,r-s}(x),
 \qquad
T_{j_1, j_2}(x) := \sum_{v \in V'} \phi_v^{(j_1)}(x) c_v \tilde g_v^{(j_2)}(x).
\end{equation}
Combining \eqref{sml5}, \eqref{deflh}, and \eqref{sml3} all together, we arrive at the relation
\begin{multline} \label{sml6}
\sum_{v \in V'} \phi_v(x) c_v \langle \tilde g_v(x), \tilde \phi(x, \la) \rangle = \sum_{s = 0}^{n-2} \hat p_s(x) \tilde \phi^{(s)}(x, \la) + \sum_{s = 0}^{n-1} t_{n,s}(x) \tilde \phi^{(s)}(x, \la) \\ + \sum_{s = 0}^{n-3} \sum_{k = s+1}^{n-2} p_k(x) t_{k,s}(x) \tilde \phi^{(s)}(x, \la)
\end{multline}
For definiteness, suppose that $\tilde p_s(x) = 0$, $s = \overline{0,n-2}$. Then $y^{[s]} = y^{(s)}$, $s = \overline{0, n}$, for the problem $\tilde{\mathcal L}$, and so
$$
\langle \tilde g_v(x), \tilde \phi(x, \la) \rangle = \sum_{s = 0}^{n-1} (-1)^{n-s-1} \tilde g_v^{(n-s-1)}(x) \tilde \phi^{(s)}(x, \la).
$$

Therefore, combining the terms at $\tilde \phi^{(s)}(x, \la)$, we obtain the formulas for finding the coefficients
\begin{equation} \label{findp}
    p_s = (-1)^{n-s-1} \sum_{v \in V'} \phi_v(x) c_v \tilde g_v^{(n-s-1)}(x) - t_{n,s}(x) - \sum_{k = s+1}^{n-2} p_k(x) t_{k,s}(x),  
\end{equation}
where $s = n-2,n-1,\ldots,1,0$. These formulas coincide with the ones for the regular case (see \cite[Lemma~2.3.7]{Yur02}).

Using the relations \eqref{findp} and \eqref{defps}, one can find $\tau_{\nu}$ for $\nu = n-2, n-3, \ldots, 1, 0$.
However, the formulas \eqref{findp} have been obtained by formal calculations. They can be used for reconstruction if the coefficients $\{ \tau_{\nu} \}_{\nu = 0}^{n-2}$ are so smooth that the series in \eqref{findp} and \eqref{deftT} converge. If the coefficients $\{ \tau_{\nu}\}_{\nu = 0}^{n-2}$ are non-smooth or even distributional, then the convergence of the series is a non-trivial question, which should be investigated separately for different classes of operators. For some classes, this question is considered in Subsections~\ref{sec:3}--\ref{sec:evenW}.

\subsection{Series convergence} \label{sec:series}

In this subsection, we prove the following auxiliary lemma.

\begin{lem} \label{lem:series}
Suppose that $j_1, j_2 \in \{ 0, 1, \ldots, n-1 \}$ and $\{ l^{(j_1 + j_2)} \xi_l \} \in l_2$. Then, there exist constants $\{ A_v \}_{v \in V'}$ such that the series
\begin{equation} \label{ser}
\sum_{v \in V'} (\phi_v^{[j_1]}(x) c_v \tilde g_v^{[j_2]}(x) - A_v)
\end{equation}
converges in $L_2(0,1)$. Moreover, if $\{ l^{(j_1 + j_2)} \xi_l\} \in l_1$, then the series
\begin{equation} \label{defT}
\sum_{v \in V'} \phi_v^{[j_1]}(x) c_v \tilde g_v^{[j_2]}(x)
\end{equation}
converges absolutely and uniformly on $[0,1]$.
\end{lem}

Here and below, the quasi-derivatives for $\phi_v(x)$ are generated by the matrix $F(x)$ and for $\tilde g_v(x)$, by $\tilde F^{\star}(x)$. In order to prove Lemma~\ref{lem:series}, we need to formulate preliminary propositions.

Consider the sector $\Gamma_1 = \left\{ \rho \in \mathbb C \colon 0 < \arg \rho < \frac{\pi}{n} \right\}$. Denote by $\{ \omega_k \}_{k = 1}^n$ the roots  of the equation $\omega^n = 1$ numbered so that
\begin{equation*}
\mbox{Re} \, (\rho \om_1) < \mbox{Re} \, (\rho \om_2) < \dots < \mbox{Re} \, (\rho \om_n), \quad \rho \in \Gamma_1.
\end{equation*}
In addition, define the extended sector
\begin{equation*} 
\Gamma_{1, h} := \left\{ \rho \in \mathbb C \colon \rho + h \exp\bigl( \tfrac{\mathrm{i} \pi}{2 n}\bigr)  \in \Gamma_1 \right\}, \quad h > 0.
\end{equation*}

In the proof of Lemma~\ref{lem:series}, we need the following proposition on the Birkhoff-type solutions of equation \eqref{eqv} with certain asymptotic behavior as $|\rho| \to \iy$.

\begin{prop}[\cite{SS20}] \label{prop:y}
For some $\rho^* > 0$, equation \eqref{eqv} has a fundamental system of solutions $\{ y_k(x, \rho) \}_{k = 1}^n$ whose quasi-derivatives $y_k^{[j]}(x, \rho)$, $k = \overline{1,n}$, $j = \overline{0, n-1}$, are continuous for $x \in [0, 1]$, $\rho \in \overline{\Gamma}_{1,h}$, $|\rho| \ge \rho^*$, analytic in $\rho \in \Gamma_{1,h}$, $|\rho| > \rho^*$ for each fixed $x \in [0, 1]$, and satisfy the relation
$$
y_k^{[j]}(x, \rho) = (\rho \om_k)^j \exp(\rho \om_k x) (1 + \zeta_{jk}(x, \rho)),
$$
where
$$
\max_{j,k,x}|\zeta_{jk}(x, \rho)| \le C(\Upsilon(\rho) + |\rho|^{-1}), \quad \rho \in \overline{\Gamma}_{1,h}, \, |\rho| \ge \rho^*,
$$
and $\Upsilon(\rho)$ fulfills the condition $\{ \Upsilon(\rho_l) \} \in l_2$
for any non-condensing sequence $\{ \rho_l \} \subset \Gamma_{1,h}$. 
\end{prop}

Consider the strip $\mathcal S_0$ defined by \eqref{defSj}. Clearly, for a suitable choice of $h$ and $c$, we have $\mathcal S_0 \subset \Gamma_{1,h}$ and $\la_{l,k,\eps} = \rho_{l,k,\eps}^n$, $\rho_{l,k,\eps} \in \mathcal S_0$ for even $(n-k)$ and for sufficiently large $l$. Further in this section, we confine ourselves to considering even $(n-k)$, since the case of odd $(n-k)$ is similar.

\begin{prop} \label{prop:Phi}
Suppose that $k \in \{ 1, 2, \ldots, n-1 \}$ and $(n-k)$ is even. Then the Weyl solution can be expanded as
$$
\Phi_{k+1}(x, \la) = \sum_{s = 1}^n b_{s,k+1}(\rho) y_s(x, \rho), \quad \la = \rho^n, \quad \rho \in \mathcal S_0,
$$
where the coefficients $b_{s,k+1}(\rho)$ are analytic in $\rho \in \mathcal S_0$, $|\rho| \ge \rho^*$, and fulfill the estimate
\begin{equation} \label{estb}
b_{s,k+1}(\rho) = O\left(\rho^{-p_{k+1,0}} \stackrel{if \: s > k+1}{\times} \exp(\rho(\om_{k+1} - \om_s)) \right).
\end{equation}
\end{prop}

\begin{proof}
The properties of the coefficient $b_{s,k+1}(\rho)$ follow from the certain formulas for these coefficients obtained in the proof of Lemma~3 in \cite{Bond21}.
\end{proof}

\begin{prop} \label{prop:ser}
Let $z$ be a non-zero complex with $\mbox{Re}\, z \le 0$, and let $\{ \varkappa_l \}_{l\ge 1} \in l_2$. Then the series $\sum\limits_{l\ge 1} \varkappa_l \exp(zlx)$ converges in $L_2(0,1)$.
\end{prop}

\begin{proof}[Proof of Lemma~\ref{lem:series}]
Let $j_1, j_2 \in \{ 0,1,\ldots,n-1\}$ be fixed.
In order to prove the convergence of the series \eqref{ser},\eqref{defT}, it is sufficient to consider their terms for $v = (l,k,\eps)$ with sufficiently large $l$. 
For technical simplicity, let us assume that $\la_{l_1,k_1,\eps} \ne \la_{l_2,k_2,\eps}$ for any sufficiently large $l_1, l_2$ such that $l_1 \ne l_2$.
In view of Corollary~\ref{cor:N}, we have
\begin{gather} \label{smser}
\sum_{v: \: l \: \text{is fixed}} \phi_v^{[j_1]}(x) c_v \tilde g_v^{[j_2]}(x) = \sum_{k = 1}^{n-1} (-1)^{n-1-p_{k,0}} \mathscr Z_{l,k}(x), \\ \nonumber \mathscr Z_{l,k}(x) := \sum_{\eps = 0,1} (-1)^{\eps} \beta_{l,k,\eps}\vv_{l,k,\eps}^{[j_1]}(x) \tilde \vv^{\star [j_2]}_{l,n-k,\eps}(x),
\end{gather}
where
\begin{gather*}
\vv_{l,k,\eps}^{[j_1]}(x) = \Phi^{[j_1]}_{k+1}(x, \la_{l,k,\eps}), \quad 
\tilde \vv_{l,n-k,\eps}^{\star [j_2]}(x) = \tilde\Phi^{\star [j_2]}_{n-k+1}(x, \la_{l,k,\eps}),
\quad \beta_{l,k,0} := \beta_{l,k}, \quad \beta_{l,k,1} := \tilde \beta_{l,k}, 
\end{gather*}

Fix $k \in \{ 1, 2, \ldots, n-1\}$ such that $(n-k)$ is even. Then, by Proposition~\ref{prop:Phi}, we have
\begin{align} \nonumber
\Phi_{k+1}^{[j_1]}(x, \la_{l,k,\eps}) & = \sum_{s_1 = 1}^n b_{s_1,k+1}(\rho_{l,k,\eps}) y_{s_1}^{[j_1]}(x, \rho_{l,k,\eps}), \\ \label{defal}
\tilde\Phi^{\star [j_2]}_{n-k+1}(x, \la_{l,k,\eps}) & = \sum_{s_2 = 1}^n \tilde b_{n-s_2+1,n-k+1}^{\star}(\rho_{l,k,\eps}) \tilde y_{n-s_2+1}^{\star [j_2]}(x, \rho_{l,k,\eps}).
\end{align}
Using the above relations and Proposition~\ref{prop:y}, we obtain
\begin{align} \nonumber
\mathscr Z_{l,k}(x) & = \sum_{s_1 = 1}^n \sum_{s_2 = 1}^n Z_{l,k,s_1,s_2}(x), \\ \nonumber
Z_{l,k,s_1,s_2}(x) & = \sum_{\eps =0,1} \al_{l,k,s_1,s_2,\eps} \exp(\rho_{l,k,\eps}(\om_{s_1} - \om_{s_2})x)(1 + \zeta_{s_1,j_1}(x,\rho_{l,k,\eps})) (1 + \tilde \zeta^{\star}_{n-s_2+1,j_2}(x, \rho_{l,k,\eps})), \\ \label{defalpha}
\al_{l,k,s_1,s_2,\eps} & := \beta_{l,k,\eps} b_{s_1,k+1}(\rho_{l,k,\eps}) b_{n-s_2+1,n-k+1}^{\star}(\rho_{l,k,\eps}) (\om_{s_1})^{j_1} (-\om_{s_2})^{j_2}
\rho_{l,k,\eps}^{j_1 + j_2}.
\end{align}
Consider the sums
\begin{align*}
Z_{l,k,s_1,s_2}(x) & = Z^1_{l,k,s_1,s_2}(x) +  Z^2_{l,k,s_1,s_2}(x) + Z^3_{l,k,s_1,s_2}(x) + Z^4_{l,k,s_1,s_2}(x), \\
Z^1_{l,k,s_1,s_2}(x) & := \sum_{\eps = 0,1} \al_{l,k,s_1,s_2,\eps} \exp(\rho_{l,k,\eps}(\om_{s_1} - \om_{s_2})x), \\
Z^2_{l,k,s_1,s_2}(x) & := \sum_{\eps = 0,1} \al_{l,k,s_1,s_2,\eps} \exp(\rho_{l,k,\eps}(\om_{s_1} - \om_{s_2})x) \zeta_{s_1,j_1}(x, \rho_{l,k,\eps}), \\
Z^3_{l,k,s_1,s_2}(x) & := \sum_{\eps = 0,1} \al_{l,k,s_1,s_2,\eps} \exp(\rho_{l,k,\eps}(\om_{s_1} - \om_{s_2})x) \tilde \zeta^{\star}_{n-s_2 + 1,j_2}(x, \rho_{l,k,\eps}), \\
Z^4_{l,k,s_1,s_2}(x) & := \sum_{\eps = 0,1} \al_{l,k,s_1,s_2,\eps} \exp(\rho_{l,k,\eps}(\om_{s_1} - \om_{s_2})x) \zeta_{s_1,j_1}(x, \rho_{l,k,\eps}) \tilde \zeta^{\star}_{n-s_2 + 1,j_2}(x, \rho_{l,k,\eps}).
\end{align*}
Thus, it is sufficient to study the convergence of the series $\sum\limits_{l \ge l_0} Z^{\nu}_{l,k,s_1,s_2}(x)$ for fixed $k$, $s_1$, $s_2$, and $\nu = \overline{1,4}$.

The asymptotics \eqref{asymptla} and \eqref{asymptbe} imply
\begin{equation} \label{smest1}
|\rho_{l,k,\eps}| \le C l, \quad |\be_{l,k,\eps}| \le C l^{n-1+p_{k+1,0}-p_{k,0}}.
\end{equation}
Using \eqref{defalpha} together with the estimates \eqref{smest1}
and \eqref{estb}, we obtain
$$
|\al_{l,k,s_1,s_2,\eps}| \le C l^{j_1 + j_2} \stackrel{if \: s_1 > k+1}{\times} \exp(\mbox{Re}\,(\om_{k+1} - \om_{s_1})r_k l) \stackrel{if \: s_2 < k}{\times} \exp(\mbox{Re}\,(\om_{s_2} - \om_k) r_k l),
$$
where $r_k := \frac{\pi}{\sin \tfrac{\pi k}{n}}$.
The relation \eqref{relxi} yields
$$
|\rho_{l,k,0} - \rho_{l,k,1}| \le C \xi_l, \quad |\be_{l,k,0} - \be_{l,k,1}| \le C \xi_l l^{n-1+p_{k+1,0}-p_{k,0}}.
$$
Since the functions $b_{s,k+1}(\rho)$ are analytic and satisfy \eqref{estb}, we obtain
$$
|b_{s,k+1}(\rho_{l,k,0}) - b_{s,k+1}(\rho_{l,k,1})| \le C \xi_l l^{-p_{k+1,0}} \stackrel{if \: s > k+1}{\times} \exp(\mbox{Re}\,(\om_s - \om_k) r_k l).
$$
It follows from \eqref{defalpha} that
\begin{align*}
\al_{l,k,s_1,s_2,0} & - \al_{l,k,s_1,s_2,1} = (\beta_{l,k,0} - \beta_{l,k,1}) b_{s_1,k+1}(\rho_{l,k,0}) b_{n-s_2+1,n-k+1}^{\star}(\rho_{l,k,0}) (\om_{s_1})^{j_1} (-\om_{s_2})^{j_2}
\rho_{l,k,0}^{j_1 + j_2} \\ & + \beta_{l,k,1} (b_{s_1,k+1}(\rho_{l,k,0}) - b_{s_1,k+1}(\rho_{l,k,1})) b_{n-s_2+1,n-k+1}^{\star}(\rho_{l,k,0}) (\om_{s_1})^{j_1} (-\om_{s_2})^{j_2}
\rho_{l,k,0}^{j_1 + j_2} \\ & + \beta_{l,k,1} b_{s_1,k+1}(\rho_{l,k,1}) (b_{n-s_2+1,n-k+1}^{\star}(\rho_{l,k,0}) - b_{n-s_2+1,n-k+1}^{\star}(\rho_{l,k,1})) (\om_{s_1})^{j_1} (-\om_{s_2})^{j_2}
\rho_{l,k,0}^{j_1 + j_2} \\ & + \beta_{l,k,1} b_{s_1,k+1}(\rho_{l,k,1}) b_{n-s_2+1,n-k+1}^{\star}(\rho_{l,k,1}) (\om_{s_1})^{j_1} (-\om_{s_2})^{j_2} (\rho_{l,k,0}^{j_1 + j_2} - \rho_{l,k,1}^{j_1 + j_2}).
\end{align*}
Consequently, we estimate
\begin{gather*}
|\al_{l,k,s_1,s_2,0} - \al_{l,k,s_1,s_2,1}| \le C l^{j_1 + j_2} \xi_l \stackrel{if \: s_1 > k+1}{\times} \exp(\mbox{Re}\,(\om_{k+1} - \om_{s_1})r_k l) \\ \stackrel{if \: s_2 < k}{\times} \exp(\mbox{Re}\,(\om_{s_2} - \om_k) r_k l).
\end{gather*}

Suppose that $\{ l^{j_1 + j_2} \xi_l \} \in l_2$. Consider the cases:
\begin{enumerate}
\item If $s_1 = s_2 \not\in \{ k, k+1 \}$, then the terms of the series $\sum\limits_{l \ge l_0} Z^1_{l,k,s_1,s_2}(x)$ decay exponentially, so the series converges absolutely.
\item If $s_1 = s_2 \in \{ k, k+1 \}$, then the series $\sum\limits_{l \ge l_0} (\al_{l,k,s_1,s_2,0} - \al_{l,k,s_1,s_2,1})$ not necessarily converges.
\item If $s_1 \ne s_2$, then
\begin{align*}
Z^1_{l,k,s_1,s_2}(x) = & ((\al_{l,k,s_1,s_2,0} - \al_{l,k,s_1,s_2,1}) \\ & + \al_{l,k,s_1,s_2,1} [(\rho_{l,k,0} - \rho_{l,k,1})(\om_{s_1} - \om_{s_2}) x + O(\xi_l^2)])\exp(\rho_{l,k,0}(\om_{s_1} - \om_{s_2})x).  
\end{align*}
Consequently, the series $\sum\limits_{l \ge l_0} Z^1_{l,k,s_1,s_2}(x)$ converges in $L_2(0,1)$ by virtue of Proposition~\ref{prop:ser}.
\end{enumerate}
Using Proposition~\ref{prop:y}, we show that
\begin{align*}
& |\zeta_{s_1,j_1}(x, \rho_{l,k,\eps})| \le C (\Upsilon(\rho_{l,k,\eps}) + l^{-1}), \\ & |\zeta_{s_1,j_1}(x, \rho_{l,k,0}) - \zeta_{s_1,j_1}(x, \rho_{l,k,1})| \le C \xi_l (\Upsilon(\rho_{l,k,0}^*) + l^{-1}),
\end{align*}
where $\Upsilon(\rho_{l,k,0}^*) = \max\limits_{|\rho - \rho_{l,k,0}|\le \de} \Upsilon(\rho)$. Note that $\{ \Upsilon(\rho^*_{l,k,0}) \} \in l_2$. Consequently, the series $\sum\limits_{l \ge l_0} Z^2_{l,k,s_1,s_2}(x)$ converges absolutely and uniformly on $[0,1]$. The proof for $Z^3$ and $Z^4$ is analogous. Thus, the regularized series $\sum\limits_{l \ge l_0} (\mathscr Z_{l,k}(x) - A_{l,k})$ converges in $L_2(0,1)$ with the constants
$$
A_{l,k} = \sum_{s = k,k+1} (\al_{l,k,s,s,0} - \al_{l,k,s,s,1}).
$$

Using the arguments above, we obtain the estimate
$$
|\mathscr Z_{l,k}(x)| \le C l^{j_1 + j_2} \xi_l.
$$
Hence, in the case $\{ l^{j_1 +j_2} \xi_l \} \in l_1$, the series $\sum\limits_{l \ge l_0} \mathscr Z_{l,k}(x)$ converges absolutely and uniformly with respect to $x \in [0,1]$. Taking \eqref{smser} into account, we arrive at the assertion of the lemma.
\end{proof}

\subsection{Case $n = 3$.} \label{sec:3}

Consider the differential expression 
$$
\ell_3(y) = y^{(3)} + (\tau_1(x) y)' + \tau_1(x) y' + \tau_0(x) y, \quad x \in (0,1),
$$
where $\tau_1 \in L_2(0,1)$ and $\tau_0 \in W_2^{-1}(0,1)$, that is, $\tau_0 = \sigma_0'$, $\sigma_0 \in L_2(0,1)$. The associated matrix has the form (see, e.g., \cite{MS19}):
\begin{equation} \label{F3}
F(x) = \begin{bmatrix}
            0 & 1 & 0 \\
            -(\sigma_0 + \tau_1) & 0 & 1 \\
            0 & (\sigma_0 - \tau_1) & 0
        \end{bmatrix},
\end{equation}
so $y^{[1]} = y'$, $y^{[2]} = y'' + (\sigma_0 + \tau_1) y$, $y^{[3]} = \ell_3(y)$.

Suppose that $p_{s,0} = s-1$, $p_{s,1} = 3-s$, $s = \overline{1,3}$, in the linear forms \eqref{defU}. Using the technique of \cite{Bond22-asympt}, we obtain the eigenvalue asymptotics
\begin{equation} \label{asymptla3}
\la_{l,k} = (-1)^{k+1}\left( \frac{2 \pi}{\sqrt 3} \Bigl( l + \frac{1}{6} + \frac{(-1)^k}{\pi^2 l} \int_0^1 \tau_1(t) \, dt + \frac{\varkappa_{l,k}}{l}\Bigr)\right)^3, \quad \{ \varkappa_{l,k} \} \in l_2, \quad l \ge 1, \: k = 1,2.
\end{equation}

Assume that $\mathcal L \in W$. It can be easily shown that, if $\Lambda_1 \cap \Lambda_2 = \varnothing$, then the spectral data $\{ \la_0, \mathcal N(\la_0) \}_{\la_0 \in \Lambda}$ do not depend on the boundary condition coefficients $u_{s,j,a}$. Therefore, let us assume that $U_0 = I$, $U_1 = [\de_{k,4-j}]_{k,j = 1}^3$. Consider the following inverse problem.

Consider the problems $\mathcal L = (F(x), U_0, U_1) \in W$ and $\tilde {\mathcal L} = (\tilde F(x), U_0, U_1) \in W$, where $\tilde F(x)$ is the matrix function associated with the differential expression $\tilde \ell_3(y)$ having the coefficients $\tilde \tau_1 \in L_2(0,1)$ and $\tilde \tau_0 = \tilde \sigma_0' \in W_2^{-1}(0,1)$. 
Under the above assumptions, the following uniqueness theorem for solution of Problem~\ref{prob:sd-coef} is valid.

\begin{thm} \label{thm:uniq3}
If $\Lambda = \tilde \Lambda$ and $\mathcal N(\la_0) = \tilde{\mathcal N}(\la_0)$ for all $\la_0 \in \Lambda$, then $\tau_1(x) = \tilde \tau_1(x)$ and
$\sigma_0(x) = \tilde \sigma_0(x) + const$ a.e. on $(0,1)$. Thus, the spectral data $\{ \la_0, \mathcal N(\la_0) \}_{\la_0 \in \Lambda}$ uniquely specify $\tau_1 \in L_2(0,1)$ and $\tau_0 \in W_2^{-1}(0,1)$.
\end{thm}

In order to prove Theorem~\ref{thm:uniq3}, we need the following auxiliary lemma, which is valid for $n$ not necessarily equal $3$.

\begin{lem} \label{lem:Pconst}
If $\mathcal L, \tilde{\mathcal L} \in W$, $\Lambda = \tilde \Lambda$ and $\mathcal N(\la_0) = \tilde{\mathcal N}(\la_0)$ for all $\la_0 \in \Lambda$, then the matrix of spectral mappings $\mathcal P(x, \la)$ defined by \eqref{defP} does not depend on $\la$.
\end{lem}

\begin{proof}
It follows from \eqref{defP} and \eqref{PJP} that
$$
\mathcal P(x, \la) = \Phi(x, \la) J_0^{-1} [\tilde \Phi^{\star}(x, \la)]^T J.
$$
Using \eqref{relN1} and \eqref{relNPhi}, we derive for $\la_0 \in \Lambda$:
\begin{align*}
\mathcal P_{\langle -2 \rangle}(x, \la) J^{-1} & = 
\Phi_{\langle -1 \rangle}(x, \la_0) J_0^{-1} [\tilde \Phi_{\langle -1 \rangle}^{\star}(x, \la_0)]^T \\ & = 
\Phi_{\langle 0 \rangle}(x, \la_0) \mathcal N(\la_0) J_0^{-1} [\mathcal N^*(\la_0)]^T[\tilde \Phi_{\langle 0 \rangle}^{\star}(x, \la_0)]^T = 0, \\
\mathcal P_{\langle -1 \rangle}(x, \la) J^{-1} & = \Phi_{\langle -1 \rangle}(x, \la_0) J_0^{-1} [\tilde \Phi_{\langle 0 \rangle}^{\star}(x, \la_0)]^T + 
\Phi_{\langle 0 \rangle}(x, \la_0) J_0^{-1} [\tilde \Phi_{\langle -1 \rangle}^{\star}(x, \la_0)]^T \\ & = \Phi_{\langle 0 \rangle}(x, \la_0)  (\mathcal N(\la_0) J_0^{-1} + J_0^{-1} [\mathcal N^{\star}(\la_0)]^T)  [\tilde \Phi_{\langle 0 \rangle}^{\star}(x, \la_0)]^T = 0.
\end{align*}
Hence, $\mathcal P(x, \la)$ is entire in $\la$. Using the asymptotics \eqref{asymptP} and Liouville's theorem, we conclude that $\mathcal P(x,\la) \equiv \mathcal P(x)$, $x \in [0,1]$.
\end{proof}

\begin{proof}[Proof of Theorem~\ref{thm:uniq3}]
This proof is similar to the proof of Theorem~2 in \cite{Bond21}, so we outline it briefly. By Lemma~\ref{lem:Pconst}, $\mathcal P(x, \la) \equiv \mathcal P(x)$. Furthermore, $\mathcal P(x)$ is a unit lower-triangular matrix.
One can easily show that
\begin{equation} \label{PF}
\mathcal P'(x) + \mathcal P(x) \tilde F(x) = F(x) \mathcal P(x), \quad x\in (0,1),
\end{equation}
where the matrix functions $F(x)$ and $\tilde F(x)$ have the form \eqref{F3}. In the element-wise form, \eqref{PF} implies $\mathcal P_{2,1} = \mathcal P_{3,2} = \mathcal P_{3,1}' = 0$, $\mathcal P_{3,1} = \hat \sigma_0 \pm \hat \tau_1$. Hence, $\hat \tau_1 = 0$, $\hat \sigma_0 = const$ in $L_2(0,1)$, which concludes the proof.
\end{proof}

Now suppose that the spectral data $\{ \la_0, \mathcal N(\la_0) \}_{\la_0 \in \Lambda}$ of the problem $\mathcal L = (F(x), U_0, U_1)$ are given. Using the asymptotics \eqref{asymptla3}, one can find the number $\tilde \tau_1 := \int_0^1 \tau_1(t) \, dt$. Put 
\begin{equation} \label{F3t}
\tilde F(x) = \begin{bmatrix}
            0 & 1 & 0 \\
            -\tilde \tau_1 & 0 & 1 \\
            0 & - \tilde \tau_1 & 0
        \end{bmatrix},
\end{equation}
and $\tilde{\mathcal L} := (\tilde F(x), U_0, U_1)$. Clearly, $\tilde F^{\star}(x) = \tilde F(x)$. Consequently, in our case,
$$
\langle \tilde g_v, \tilde \phi \rangle = \tilde g''_v \tilde \phi - \tilde g_v' \tilde \phi' + \tilde g_v \tilde \phi'' + 2 \tilde \tau_1 \tilde g_v \tilde \phi.
$$
Hence, the relation \eqref{sml5} takes the form
\begin{multline*}
    T_{0,0} \tilde \phi'' - T_{0,1} \tilde \phi' + (T_{0,2} + 2 \tilde \tau_1 T_{0,0}) \tilde \phi \\ = 
    \hat \tau_1 \tilde \phi' + (\hat \tau_1' + \hat \tau_0) \tilde \phi + T_{0,0} \tilde \phi'' + (3 T_{1,0} + 2 T_{0,1}) \tilde \phi' + (3 T_{2,0} + 3 T_{1,1} + T_{0,2} + 2 \tau_1 T_{0,0}) \tilde \phi,
\end{multline*}
where $T_{j_1,j_2}$ were defined in \eqref{deftT}. 
Grouping the terms at $\tilde \phi'(x, \la)$ and $\tilde \phi(x, \la)$, we derive the formulas 
\begin{align*} 
\tau_1 & = \tilde \tau_1 -\frac{3}{2} \sum_{v \in V'} (\phi_v' c_v \tilde g_v + \phi_v c_v \tilde g_v'),\\ 
\tau_0 & = -\hat \tau_1' - 3 \frac{d}{dx}\left( \sum_{v \in V'} \phi_v' c_v \tilde g_v \right) - 2 \hat \tau_1 \sum_{v \in V'} \phi_v c_v \tilde g_v.
\end{align*}

By virtue of Corollary~1.3 and Theorem~6.4 from \cite{Bond22-asympt} and \eqref{relxi}, we have $\{ l \xi_l \} \in l_2$. Applying Lemma~\ref{lem:series} to prove the series convergence in suitable spaces and using the notations \eqref{defeta}, we arrive at the following reconstruction formulas for $\tau_1$ and $\tau_0$.

\begin{thm}
Let $\mathcal L$ and $\tilde{\mathcal L}$ be the problems defined above in this section. The following relations hold:
\begin{align} \label{rectau1}
    \tau_1 & = \tilde \tau_1 -\frac{3}{2} \sum_{(l,k,\eps) \in V} (\vv'_{l,k,\eps} \tilde \eta_{l,k,\eps} + \vv_{l,k,\eps} \tilde \eta'_{l,k,\eps}), \\ \label{rectau0}
\tau_0 & = -\hat \tau_1' - 3 \frac{d}{dx}\left( \sum_{(l,k,\eps)\in V} \vv'_{l,k,\eps} \tilde \eta_{l,k,\eps} \right) - 2 \hat \tau_1 \sum_{(l,k,\eps)\in V} \vv_{l,k,\eps} \tilde \eta_{l,k,\eps}.
\end{align}
The series in \eqref{rectau1} converges in $L_2(0,1)$. In \eqref{rectau0}, the series in brackets converges in $L_2(0,1)$ with regularization, and the second series converges absolutely and uniformly with respect to $x \in [0,1]$, so the right-hand side of \eqref{rectau0} belongs to $W_2^{-1}(0,1)$.
\end{thm}

Following the proof of Lemma~\ref{lem:series}, one
can easily show that the regularization constants $A_v$ for the series in \eqref{rectau1} equal zero.
The regularization constants in \eqref{rectau0} are omitted because of the differentiation. Finally, we arrive at the following algorithm for solving Problem~\ref{prob:sd-coef}.

\begin{alg}
Suppose that the spectral data $\{ \la_0, \mathcal N(\la_0) \}_{\la_0 \in \Lambda}$ of the problem $\mathcal L = \mathcal L(F(x), U_0, U_1) \in W$ are given. Here $F(x)$ is defined by \eqref{F3}, $U_0 = I$, $U_1 = [\de_{k,4-j}]_{k,j = 1}^3$. We have to find $\tau_1$ and $\tau_0$.

\begin{enumerate}
\item Find $\tilde \tau_1 = \int_0^1 \tau_1(x) \,dx$ from the eigenvalue asymptotics \eqref{asymptla3}.
\item Take the model problem $\tilde {\mathcal L} = \mathcal L(\tilde F(x), U_0, U_1)$, where $\tilde F(x)$ is defined by \eqref{F3t}.
\item Implement the steps 2--6 of Algorithm~\ref{alg:1} to obtain $\{ \vv_{l,k,\eps}(x) \}_{(l,k,\eps) \in V}$.
\item Using the problem $\tilde {\mathcal L}$ and the spectral data $\{ \la_0, \mathcal N(\la_0) \}_{\la_0 \in \Lambda}$, $\{ \tilde \la_0, \tilde{\mathcal N}(\tilde \la_0) \}_{\tilde \la_0 \in \tilde \Lambda}$, construct the functions $\{ \tilde \eta_{l,k,\eps}(x) \}_{(l,k,\eps) \in V}$ by \eqref{defeta}.
\item Construct $\tau_1(x)$ and $\tau_0(x)$ by \eqref{rectau1} and \eqref{rectau0}, respectively.
\end{enumerate}
\end{alg}

\subsection{Case of even $n$, $\tau_{\nu} \in L_2(0,1)$.} \label{sec:evenL}

Consider the differential expression \eqref{defl} with even $n$ and $\tau_{\nu} \in L_2(0,1)$, $\nu = \overline{0, n-2}$. The associated matrix $F(x) = [f_{k,j}(x)]_{k,j = 1}^n$ is given by the relations
\begin{align*}
& f_{n-k, k+1} = -\tau_{2k}, \quad k = \overline{0,\lfloor n/2\rfloor-1}, \\
& f_{n-k-1,k+1} = f_{n-k,k+2} = -\tau_{2k+1}, \quad k = \overline{0,\lfloor n/2 \rfloor - 2},
\end{align*}
and all the other elements are defined by $f_{k,j} = \de_{k,j-1}$. For instance, 
$$
\ell_6(y) = y^{(6)} + (\tau_4 y'')'' + [(\tau_3 y'')' + (\tau_3 y')''] + (\tau_2 y')' + [(\tau_1 y)' + \tau_1 y'] + \tau_0 y,
$$
and the corresponding associated matrix is
$$
F(x) = \begin{bmatrix}
            0 & 1 & 0 & 0 & 0 & 0 \\
            0 & 0 & 1 & 0 & 0 & 0 \\
            0 & 0 & 0 & 1 & 0 & 0 \\
            0 & -\tau_3 & -\tau_4 & 0 & 1 & 0 \\
            -\tau_1 & -\tau_2 & -\tau_3 & 0 & 0 & 1 \\
            -\tau_0 & -\tau_1 & 0 & 0 & 0 & 0 
        \end{bmatrix}.
$$

Suppose that $U_0 = I$, $U_1 = [\de_{k,n-j+1}]_{k,j = 1}^n$, $\mathcal L = (F(x), U_0, U_1) \in W$ and $\tilde{\mathcal L} = (\tilde F(x), U_0, U_1)$, where $\tilde F(x)$ is constructed in the same way as $F(x)$ by different coefficients $\tilde \tau_{\nu} \in L_2(0,1)$, $\nu = \overline{0, n-2}$. The following uniqueness theorem is proved similarly to Theorem~\ref{thm:uniq3}.

\begin{thm}
If $\Lambda = \tilde \Lambda$ and $\mathcal N(\la_0) = \tilde {\mathcal N}(\la_0)$ for all $\la_0 \in \Lambda$, then $\tau_{\nu}(x) = \tilde \tau_{\nu}(x)$ a.e. on $(0,1)$, $\nu = \overline{0,n-2}$. Thus, the spectral data $\{ \la_0, \mathcal N(\la_0) \}_{\la_0 \in \Lambda}$ uniquely specify $\tau_{\nu} \in L_2(0,1)$, $\nu = \overline{0,n-2}$.
\end{thm}

Further, we need the following proposition, which is an immediate corollary of Theorems~1.2 and~6.4 from \cite{Bond22-asympt} for the problems $\mathcal L$, $\tilde{\mathcal L}$ defined above in this subsection and the sequence $\{ \xi_l \}$ defined by \eqref{defxi} (see also Example~5.2 in \cite{Bond22-asympt}).

\begin{prop}[\cite{Bond22-asympt}] \label{prop:difL}
Suppose that $\nu_0 \in \{ 1, 2, \ldots, n-1 \}$, $\tau_{\nu}(x) = \tilde \tau_{\nu}(x)$ a.e. on $(0,1)$ for $\nu = \overline{\nu_0, n-2}$, and $\int_0^1 \hat \tau_{\nu_0-1}(x) \, dx = 0$. Then $\{ l^{n - \nu_0} \xi_l \} \in l_2$.
\end{prop}

We will construct the solution of Problem~\ref{prob:sd-coef} step-by-step.

\medskip

\textsc{Step 1.} Take the model problem $\tilde {\mathcal L} = \tilde {\mathcal L}^{(1)} := (\tilde F^{(1)}(x), U_0, U_1)$, where $\tilde F^{(1)}(x)$ is the associated matrix for the differential expression $\tilde l_n^{(1)}(y)$ with the coefficients $\tilde \tau_{n-2} := \int_0^1 \tau_{n-2}(x) \, dx$, $\tilde \tau_\nu := 0$, $\nu = \overline{0,n-3}$. The coefficient $\int_0^1 \tau_{n-2}(x) \, dx$ can be found from the eigenvalue asymptotics similarly to the case of Subsection~\ref{sec:3}. Using the terms of \eqref{sml5} at $\tilde \phi^{(n-2)}(x, \la)$, we derive the reconstruction formula
$$
\tau_{n-2} = \tilde \tau_{n-2} - t_{n,n-2} - T_{0,1} =
\tilde \tau_{n-2} - n \sum_{v \in V'} (\phi_v' c_v \tilde g_v + \phi_v c_v \tilde g_v').
$$
By virtue of Proposition~\ref{prop:difL}, $\{ l \xi_l \} \in l_2$. Therefore, Lemma~\ref{lem:series} implies that the obtained series converges in $L_2(0,1)$ with the regularization constants $A_v = 0$.

\smallskip

\textsc{Step 2.} Take the model problem $\tilde {\mathcal L} = \tilde {\mathcal L}^{(2)} := (\tilde F^{(2)}(x), U_0, U_1)$, where $\tilde F^{(2)}(x)$ is the associated matrix for the differential expression $\tilde l_n^{(2)}(y)$ with the coefficients $\tilde \tau_{n-2} := \tau_{n-2}$, $\tilde \tau_{n-3} := 
\int_0^1 \tau_{n-3}(x) \, dx$, $\tilde \tau_\nu := 0$, $\nu = \overline{0,n-4}$. The coefficient $\int_0^1 \tau_{n-3}(x) \, dx$ can be found from the eigenvalue asymptotics. Using the terms of \eqref{sml5} at $\tilde \phi^{(n-2)}(x, \la)$, we show that $T_{0,0}'(x) = 0$. One can easily show that $T_{0,0}(0) = 0$, so $T_{0,0}(x) \equiv 0$. Consequently, grouping the terms of
\eqref{sml5} at $\tilde \phi^{(n-3)}(x, \la)$, we obtain
\begin{align*}
2 \tau_{n-3} & = 2 \tilde \tau_{n-3} - t_{n,n-3} + T_{0,2} \\
& = 2 \tilde \tau_{n-3} - \sum_{v \in V'} \left( \tfrac{n(n-1)}{2} \phi''_v c_v \tilde g_v + n(n-2) \phi'_v c_v \tilde g_v' + [\tfrac{(n-1)(n-2)}{2}-1] \phi_v c_v g_v''\right).
\end{align*}
By virtue of Proposition~\ref{prop:difL}, $\{ l^2 \xi_l \} \in l_2$. Lemma~\ref{lem:series} implies that the series converges in $L_2(0,1)$ with the zero regularization constants.

\smallskip

\textsc{Step $s$}. Take the model problem $\tilde{\mathcal L} = \tilde {\mathcal L}^{(s)} := (\tilde F^{(s)}(x), U_0, U_1)$, where $\tilde F^{(s)}(x)$ is the associated matrix for the differential expression $\tilde l^{(s)}_n(y)$ with 
\begin{equation} \label{tL}
\tilde \tau_{\nu} := \tau_{\nu}, \: \nu = \overline{n-s,n-2}, \quad \tilde \tau_{n-s-1} := \int_0^1 \tau_{n-s-1}(x) \, dx, \quad \tilde \tau_{\nu} := 0, \: \nu = \overline{0, n-s-2}. 
\end{equation}
For this model problem, we have $T_{j_1, j_2}(x) \equiv 0$ for all $j_1 + j_2 \le s-2$. Grouping the terms of \eqref{sml5} at $\tilde \phi^{(n-s-1)}(x,\la)$, we obtain
\begin{align} \nonumber
    \tau_{n-s-1} & = \tilde \tau_{n-s-1} - (t_{n,n-s-1} + (-1)^{s+1} T_{0,s}) \stackrel{if \: s \: is \: even}  {\times} \tfrac{1}{2} \\ \nonumber
    & = \tilde \tau_{n-s-1} - \sum_{v \in V'} \left(\sum_{r = n-s}^n C_n^r C_{r-1}^{n-s-1} \phi_v^{(n-r)} c_v \tilde g_v^{(r-n+s)} + (-1)^{s+1} \phi_v c_v \tilde g_v^{(s)} \right) \stackrel{if \: s \: is \: even}  {\times} \tfrac{1}{2} \\ \label{rectau}
    & = \tilde \tau_{n-s-1} - \sum_{v \in V'} \left(\sum_{r = n-s}^n C_n^r C_{r-1}^{n-s-1} \phi_v^{[n-r]} c_v \tilde g_v^{[r-n+s]} + (-1)^{s+1} \phi_v c_v \tilde g_v^{[s]} \right) \stackrel{if \: s \: is \: even}  {\times} \tfrac{1}{2}
\end{align}
Proposition~\ref{prop:difL} implies that $\{ l^s \xi_l \} \in l_2$. Therefore, it follows from Lemma~\ref{lem:series} that the series in \eqref{rectau} converges in $L_2(0,1)$. The regularization constants equal zero because
$$
\sum_{r = n-s}^n C_n^r C_{r-1}^{n-s-1} (-1)^r + (-1)^{s+1} = 0.
$$
Note that all functions $\{ \tau_{\nu} \}$ necessary for computation of the quasi-derivatives $\phi_v^{[n-r]}$ in \eqref{rectau} are computed at the previous steps, so the formula \eqref{rectau} can be used for finding $\tau_{n-s-1}$.
In terms of the notations \eqref{defeta}, the relation \eqref{rectau} can be written as follows:
\begin{equation} \label{rec}
\tau_{n-s-1} = \tilde \tau_{n-s-1} - \sum_{(l,k,\eps) \in V} \left(\sum_{r = n-s}^n C_n^r C_{r-1}^{n-s-1} \vv_{l,k,\eps}^{[n-r]} \tilde \eta_{l,k,\eps}^{[r-n+s]} + (-1)^{s+1} \vv_{l,k,\eps} \tilde \eta_{l,k,\eps}^{[s]} \right) \stackrel{if \: s \: is \: even}  {\times} \tfrac{1}{2}
\end{equation}

Thus, we obtain the following algorithm for solving Problem~\ref{prob:sd-coef} in the considered case.

\begin{alg} \label{alg:even}
Suppose that the spectral data $\{ \la_0, \mathcal N(\la_0) \}_{\la_0 \in \Lambda}$ of the problem $\mathcal L = (F(x),U_0,U_1) \in W$ are given. Here $F(x)$ is the matrix associated with the differential expression $\ell_n(y)$, $n$ is even, $\tau_{\nu} \in L_2(0,1)$, $\nu = \overline{0,n-2}$, $U_0 = I$, $U_1 = [\de_{k,n-j+1}]_{k,j=1}^n$. We have to find $\{ \tau_{\nu} \}_{\nu = 0}^{n-2}$. For simplicity, assume that the values $\int_0^1 \tau_{\nu}(x) \, dx$ are known. In fact, they can be found from the eigenvalue asymptotics.

For $s = 1, 2, \ldots, n-1$, we find $\tau_{n-s-1}$ implementing the following steps:

\begin{enumerate}
    \item Take the model problem $\tilde {\mathcal L} = \tilde {\mathcal L}^{(s)} = (\tilde F^{(s)}, U_0, U_1)$ induced by the differential expression $\tilde \ell^{(s)}_n(y)$ with the coefficients \eqref{tL}.
    \item Implement steps 2--6 of Algorithm~\ref{alg:1} to find $\{ \vv_{l,k,\eps}(x) \}_{(l,k,\eps) \in V}$.
    \item Using the problem $\tilde {\mathcal L}$ and the spectral data $\{ \la_0, \mathcal N(\la_0) \}_{\la_0 \in \Lambda}$, $\{ \tilde \la_0, \tilde{\mathcal N}(\tilde \la_0) \}_{\tilde \la_0 \in \tilde \Lambda}$, construct the functions $\{ \tilde \eta_{l,k,\eps}(x) \}_{(l,k,\eps) \in V}$ by \eqref{defeta}.
    \item Construct $\tau_{n-s-1}(x)$ by \eqref{rec}.
\end{enumerate}
\end{alg}

\subsection{Case of even $n$, $\tau_{\nu} \in W_2^{-1}(0,1)$} \label{sec:evenW}

Suppose that $n$ is even and $\tau_{\nu} \in W_2^{-1}(0,1)$ in \eqref{defl} for $\nu = \overline{0, n-2}$, that is, $\tau_{\nu} = \sigma_{\nu}'$, where $\sigma_{\nu} \in L_2(0,1)$ and the derivative is understood in the sense of distributions. Put $m := \lfloor n/2 \rfloor$ and define the matrix function
$$
Q(x) = [q_{r,j}(x)]_{r,j = 0}^m := \sum_{\nu = 0}^{n-2} (-1)^{\lfloor (\nu - 1)/2\rfloor} \chi_{\nu} \sigma_{\nu}(x), 
$$
where $\chi_{\nu} := [\chi_{\nu;r,j}]_{r,j = 0}^m$,
$$
\chi_{2k;k,k+1} = \chi_{2k;k+1,k} = 1, \quad
\chi_{2k+1;k,k+2} = -\chi_{2k+1;k+2,k} = 1,
$$
and all the other entries $\chi_{\nu;r,j}$ equal zero. The associated matrix $F(x) = [f_{k,j}(x)]_{k,j = 1}^n$ for $\ell_n(y)$ is defined as follows (see \cite{Bond22} for details):
\begin{gather*}
f_{m,j} := (-1)^{m+1} q_{j-1,m}, \: j = \overline{1, m}, \qquad
f_{k,m+1} := (-1)^{k+1} q_{m,2m-k}, \: k = \overline{m+1, 2m}, \\
f_{k,j} := (-1)^{k+1} q_{j-1,2m-k} + (-1)^{m+k} q_{j-1,m} q_{m,2m-k}, \quad k = \overline{m+1,2m}, \, j = \overline{1,m},
\end{gather*}  
and $f_{k,j} = \de_{k,j-1}$ for all the other indices. Clearly, $F \in \mathfrak F_n$. For example, if $n = 4$, then
$$
Q(x) = \begin{bmatrix}
        0 & -\sigma_0 & \sigma_1 \\
        -\sigma_0 & 0 & \sigma_2 \\
        -\sigma_1 & \sigma_2 & 0
      \end{bmatrix}, \quad
F(x) = \begin{bmatrix}
        0 & 1 & 0 & 0 \\
        -\sigma_1 & -\sigma_2 & 1 & 0 \\
        -\sigma_0 - \sigma_1 \sigma_2 & -\sigma_2^2 & \sigma_2 & 1 \\
        -\sigma_1^2 & \sigma_0 - \sigma_1 \sigma_2 & \sigma_1 & 0
    \end{bmatrix}
$$

Consider Problem~\ref{prob:sd-coef} for $\mathcal L = (F(x), U_0, U_1)$, $U_0 = I$, $U_1 = [\de_{k,n-j+1}]_{k,j=1}^n$. Let $\tilde{\mathcal L} = (\tilde F(x), U_0, U_1)$, where $\tilde F(x)$ is the associated matrix for the differential expression $\tilde l_n(y)$ with the coefficients
$\tilde \tau_{\nu} = \tilde \sigma_{\nu}' \in W_2^{-1}(0,1)$, $\nu = \overline{0,n-2}$. The following uniqueness theorem is proved analogously to Theorem~\ref{thm:uniq3}.

\begin{thm}
If $\Lambda = \tilde \Lambda$ and $\mathcal N(\la_0) = \tilde{\mathcal N}(\la_0)$ for all $\la_0 \in \Lambda$, then $\sigma_{\nu}(x) = \tilde \sigma_{\nu}(x) + const$ a.e. on $(0,1)$ for $\nu = \overline{0,n-2}$. Thus, the spectral data $\{ \la_0, \mathcal N(\la_0) \}_{\la_0 \in \Lambda}$ uniquely specify $\tau_{\nu} \in W_2^{-1}(0,1)$, $\nu = \overline{0,n-2}$.
\end{thm}

The functions $\{ \sigma_{\nu} \}_{\nu = 0}^{n-2}$ are specified uniquely up to a constant, so for simplicity we assume that $\int_0^1 \sigma_{\nu}(x) \, dx = 0$, $\nu = \overline{0,n-2}$. 

Theorems~1.2 and~6.4 of \cite{Bond22-asympt} (see also Example~5.3 in \cite{Bond22-asympt}) readily imply the following proposition for the problems $\mathcal L$ and $\tilde{\mathcal L}$ of the considered form and the sequence $\{ \xi_l \}$ defined by \eqref{defxi}.

\begin{prop}[\cite{Bond22-asympt}] \label{prop:difW}
Suppose that $\nu_0 \in \{ 1, 2, \ldots, n-1 \}$ and $\sigma_{\nu}(x) = \tilde \sigma_{\nu}(x)$ a.e. on $(0,1)$ for $\nu = \overline{\nu_0, n-2}$. Then $\{ l^{n - \nu_0 - 1} \xi_l \} \in l_2$.
\end{prop}

The algorithm of recovering the coefficients $\{ \tau_{\nu} \}_{\nu = 0}^{n-2}$ from the spectral data is similar to Algorithm~\ref{alg:even}. At \textsc{Step $s$}, we take the model problem $\tilde {\mathcal L} = \tilde {\mathcal L}^{(s)}$ induced by the coefficients $\tilde \sigma_{\nu} := \sigma_{\nu}$, $\nu = \overline{n-s,n-2}$, and $\tilde \sigma_{\nu} := 0$, $\nu = \overline{0, n-s-1}$. Note that the series in \eqref{rectau} has the form
$$
a_0 T_{s,0} + a_1 T_{s-1,1} + \ldots + a_s T_{0,s} = (b_0 T_{s-1,0} + b_1 T_{s-2,1} + \ldots + b_{s-1} T_{0,s-1})',
$$
where
\begin{gather*}
a_j := C_n^{s-j} C_{n-s+j-1}^j \stackrel{if \: j = s}{+} (-1)^{s+1}, \quad \sum_{j = 0}^s a_j = 0, \\
b_j := \sum_{i = 0}^j (-1)^{j-i} a_i, \quad j = \overline{0,s-1}.
\end{gather*}
Using this idea, we derive
\begin{equation} \label{rec2}
\tau_{n-s-1} = -\frac{d}{dx} \sum_{v \in V'} \left( \sum_{j = 0}^{s-1} b_j \phi_v^{[s-j-1]} c_v \tilde g_v^{[j]}\right) \stackrel{if \: s \: is \: even}{\times} \tfrac{1}{2}.
\end{equation}
In view of Proposition~\ref{prop:difW}, we have $\{ l^{s-1} \xi_l \} \in l_2$. Hence, by virtue of Lemma~\ref{lem:series}, the series in \eqref{rec2} converges in $L_2(0,1)$ with some regularization constants. Because of the differentiation, we omit these constants. Thus, formula \eqref{rec2} induces a function of $W_2^{-1}(0,1)$,
and $\sigma_{n-s-1}$ can be found uniquely up to a constant. This constant is chosen so that $\int_0^1 \sigma_{n-s-1}(x) \, dx = 0$. Taking $s = 1, 2, \ldots, n-1$, we step-by-step construct all the coefficients $\tau_{n-2}$, $\tau_{n-3}$, \ldots, $\tau_1$, $\tau_0$.

Note that the algorithms of this section are valid for $\tilde{\mathcal L} \in W$. However, the case $\tilde{\mathcal L} \not\in W$ requires only technical modifications due to Remark~\ref{rem:mult}, which do not influence on the convergence of the series.

\section{Conclusion} \label{sec:concl}

Let us briefly summarize the results of this paper. We have studied the inverse spectral problem which consists in recovering the coefficients $\{ \tau_{\nu}\}_{\nu = 0}^{n-2}$ of the differential expression \eqref{defl} from the spectral data $\{ \la_0, \mathcal N(\la_0)\}_{\la_0 \in \Lambda}$. An approach to constructive solution of the inverse problem is developed. Our approach can be applied to a wide class of differential expressions $\ell_n(y)$, which admit regularization in terms of associated matrix. 

The inverse problem solution consists of the two steps. First, we consider the auxiliary problem of finding the Weyl solutions $\{ \Phi_k(x, \la)\}_{k = 1}^n$ by using the spectral data. This problem is reduced to the linear equation \eqref{main} in the Banach space $m$ of bounded infinite sequences. Theorem~\ref{thm:main} on the unique solvability of the main equation \eqref{main} is proved. Second, by using the solution of the main equation, we derive reconstruction formulas for the coefficients $\{ \tau_{\nu}\}_{\nu = 0}^{n-2}$ and investigate the convergence of resulting series.

\smallskip

Let us mention the most important \textbf{advantages} of our approach:
\begin{enumerate}
	\item The obtained results can be applied to a wide range of differential operators of arbitrary order with either integrable or distributional coefficients of various classes.
	\item Our approach does not require self-adjointness.
	\item Our method is constructive.
	\item The results of this paper can be used for studying existence and stability of the inverse problem solution as well as for developing numerical methods.
\end{enumerate}

\medskip

{\bf Funding.} This work was supported by Grant 21-71-10001 of the Russian Science Foundation, https://rscf.ru/en/project/21-71-10001/.

\noindent Natalia Pavlovna Bondarenko \\
1. Department of Applied Mathematics and Physics, Samara National Research University, \\
Moskovskoye Shosse 34, Samara 443086, Russia, \\
2. Department of Mechanics and Mathematics, Saratov State University, \\
Astrakhanskaya 83, Saratov 410012, Russia, \\
e-mail: {\it BondarenkoNP@info.sgu.ru}

\end{document}